\documentclass[pdflatex,sn-mathphys-num]{sn-jnl}

\usepackage{graphicx}%
\usepackage{multirow}%
\usepackage{amsmath,amssymb,amsfonts}%
\usepackage{amsthm}%
\usepackage{mathrsfs}%
\usepackage[title]{appendix}%
\usepackage{xcolor}%
\usepackage{textcomp}%
\usepackage{manyfoot}%
\usepackage{booktabs}%
\usepackage{algorithm}%
\usepackage{algorithmicx}%
\usepackage{algpseudocode}%
\usepackage{listings}%

\usepackage{tikz-cd}
\usepackage{relsize}
\usepackage{mathtools,xparse}

\usepackage{hyperref}
\usepackage{accents}

\usepackage{calligra}
\usepackage{enumitem}

\usepackage[T1]{fontenc}
\usepackage[ttdefault=true]{AnonymousPro}

\setlist[enumerate]{wide=\parindent}

\usepackage{bbm}

 \allowdisplaybreaks

\DeclareMathAlphabet{\mathcalligra}{T1}{calligra}{m}{n}

\DeclareMathOperator{\Res}{Res}

\DeclareMathOperator{\Hom}{Hom}

\DeclareMathOperator{\Sing}{Sing}

\DeclareMathOperator{\Aut}{Aut}

\DeclareMathOperator{\End}{End}

\DeclareMathOperator{\obj}{obj}

\DeclareMathOperator{\fgr}{gr}

\theoremstyle{thmstyleone}%
\newtheorem{thm}{Theorem}[section]
\newtheorem{prop}[thm]{Proposition}
\newtheorem{coro}[thm]{Corollary}

\newtheorem{lem}[thm]{Lemma}

\theoremstyle{thmstyletwo}%

\newtheorem{rem}[thm]{Remark}

\theoremstyle{thmstylethree}%

\theoremstyle{definition}
\newtheorem{de}[thm]{Definition}

\begin{document}


\newcommand{\C}{{\mathbb{C}}}
\newcommand{\Z}{{\mathbb{Z}}}
\newcommand{\N}{{\mathbb{N}}}
\newcommand{\te}[1]{\mbox{#1}}
\newcommand{\set}[2]{{
    \left\{
        {#1}
    \,\middle\vert\,
        {#2}
    \right\}
}}

\newcommand{\choice}[2]{{
\left[
\begin{array}{c}
{#1}\\{#2}
\end{array}
\right]
}}
\def \<{\langle}
\def \>{\rangle}

\def \:{\mathopen{\overset{\circ}{
    \mathsmaller{\mathsmaller{\circ}}}
    }}
\def \;{\mathclose{\overset{\circ}{\mathsmaller{\mathsmaller{\circ}}}}}

\newcommand{\overit}[2]{{
    \mathop{{#1}}\limits^{{#2}}
}}
\newcommand{\belowit}[2]{{
    \mathop{{#1}}\limits_{{#2}}
}}

\newcommand{\wt}[1]{\widetilde{#1}}

\newcommand{\wh}[1]{\widehat{#1}}

\newcommand{\wck}[1]{\reallywidecheck{#1}}

\newcommand{\qqq}{{\mathbbm q}}

\newlength{\dhatheight}
\newcommand{\dwidehat}[1]{%
    \settoheight{\dhatheight}{\ensuremath{\widehat{#1}}}%
    \addtolength{\dhatheight}{-0.45ex}%
    \widehat{\vphantom{\rule{1pt}{\dhatheight}}%
    \smash{\widehat{#1}}}}
\newcommand{\dhat}[1]{%
    \settoheight{\dhatheight}{\ensuremath{\hat{#1}}}%
    \addtolength{\dhatheight}{-0.35ex}%
    \hat{\vphantom{\rule{1pt}{\dhatheight}}%
    \smash{\hat{#1}}}}

\newcommand{\dwh}[1]{\dwidehat{#1}}

\newcommand{\dis}{\displaystyle}

\newcommand{\pd}[1]{\frac{\partial}{\partial {#1}}}

\newcommand{\pdiff}[2]{\frac{\partial^{#2}}{\partial #1^{#2}}}

\newcommand{\qb}[2]{\genfrac{[}{]}{0pt}{}{#1}{#2}}

\newcommand{\vvp}[1]{\< #1 \>_{\,\,\,\mathllap{\wp}}}

\newcommand{\vvinf}[1]{\< #1 \>_{\,\,\,\,\,\mathllap{\infty}}}


\newcommand{\g}{{\mathfrak g}}
\newcommand{\h}{{\mathfrak h}}
\newcommand{\al}{\alpha}
\newcommand{\be}{\beta}
\newcommand{\ssl}{{\mathfrak{sl}}}
\newcommand{\id}{\te{id}}
\newcommand{\rtu}{{\xi}}
\newcommand{\half}{{\frac{1}{2}}}
\newcommand{\reciprocal}[1]{{\frac{1}{#1}}}
\newcommand{\inverse}{^{-1}}
\newcommand{\inv}{\inverse}
\newcommand{\Apsc}[1]{A}

\newcommand{\Bpsc}[1]{B}

\newcommand{\tpsc}[2]{\wt{#1}}


\newcommand{\E}{{\mathcal{E}}}

\newcommand{\ot}{\otimes}


\newcommand{\U}{{\mathcal{U}}}
\newcommand{\vac}{{{\mathbbm 1}}}


\makeatletter
\@addtoreset{equation}{section}
\def\theequation{\thesection.\arabic{equation}}
\makeatother \makeatletter

\title[Quantum affine VA]{Filtration and gradation of quantum affine vertex algebras}

\author*[1]{\fnm{Fei} \sur{Kong}
}
\email{kongmath@hunnu.edu.cn}

\affil*[1]{\orgdiv{Key Laboratory of Computing and Stochastic Mathematics (Ministry of Education), School of Mathematics and Statistics}, \orgname{Hunan Normal University}, \orgaddress{\city{Changsha}, \postcode{410081}, \country{China}}}



\abstract{In this paper, we present a unified construction of quantum affine vertex algebras associated to double Yangians, untwisted quantum affinization algebras, and twisted quantum affinization algebras. We show that these $\hbar$-adic quantum vertex algebras contain a dense
quantum vertex subalgebra equipped with a increasing filtration. Moreover, we prove that the associated graded algebras with respect to these filtrations are isomorphic to a $\Z$-graded dense quantum vertex subalgebra of the quantum affine vertex algebras associated to double Yangians.}



%
\maketitle


\section{Introduction}

It is well known that vertex algebras are naturally associated with untwisted affine Lie algebras, while twisted affine Lie algebras are linked to the same vertex algebras through the theory of twisted modules (or quasi-modules).
In the general vertex algebra theory, an important subject arguably is establishing and exploring natural associations of quantum vertex algebras to double Yangians and (untwisted or twisted) quantum affine algebras.

As one of the fundamental works, Etingof and Kazhdan (\cite{EK-qva}) developed a theory of
quantum vertex operator algebras
in sense of formal deformations of vertex algebras,
and constructed a family of examples as formal
deformations of the (universal) affine vertex algebras of type $A$ by using the $R$-matrix type relations given in \cite{RS-RTT}.
The most visible difference between quantum vertex operator algebras and vertex algebras is that the usual locality is replaced by the $S$-locality.
This $S$-locality is governed by an operator called the quantum Yang-Baxter operator, which satisfies the rational quantum Yang-Baxter equation.

In \cite{Li-nonlocal, Li-constructing}, the notions of nonlocal vertex algebra and (weak) quantum vertex algebra
were introduced and systematically studied; nonlocal vertex algebras serve as vertex algebra analogues of
noncommutative associative algebras.
A nonlocal vertex algebra is a weak quantum vertex algebra \cite{Li-nonlocal} if it satisfies the $S$-locality. In addition, it becomes a quantum vertex algebra \cite{Li-nonlocal} if the $S$-locality is controlled by a rational quantum Yang-Baxter operator.
As slight generalizations of quantum vertex operator algebras in the sense of Etingof-Kazhdan,
$\hbar$-adic (weak) quantum vertex algebras, which are formal deformations of (weak) quantum vertex algebras,
were studied in \cite{Li-h-adic}.
In particular, a conceptual construction of $\hbar$-adic weak quantum vertex algebras
and their modules was established from ``$S$-local'' sets of fields (vertex operators).
As an application, $\hbar$-adic quantum vertex algebras were associated to the centrally extended double Yangian of $\mathfrak{sl}_2$.
Mainly in order to associate quantum vertex algebras to quantum affine algebras, a theory of $\phi$-coordinated quasi-modules was developed in \cite{Li-phi-coor, Li-G-phi}.

In \cite{BJK-qva-BCD},
Butorac, Jing and Ko\v{z}i\'{c} extended Etingof-Kazhdan's construction
to type $B$, $C$ and $D$ rational $R$-matrices.
The modules of these quantum vertex operator algebras are in one-to-one correspondence with restricted modules for the corresponding double Yangians (see \cite{K-qva-phi-mod-BCD}).
For any symmetrizable generalized Cartan matrix $A$, we introduced an algebra $\wh{\mathcal{DY}}(A)$ in \cite{KL-YD-1}, which is essentially the centrally extended double Yangian when $A$ is of finite type, and we gave a Drinfeld-type presentation of $\wh{\mathcal{DY}}(A)$.
Using this, we constructed an $\hbar$-adic weak quantum vertex algebra $\mathcal V_A(\ell)$ for each $\ell\in\C$, and established an isomorphism between the category of restricted $\wh{\mathcal{DY}}(A)$-modules of level $\ell$ and the category of $\mathcal V_A(\ell)$-modules.

Based on the $R$-matrix presentation of quantum affine algebras
(see \cite{DF-qaff-RTT-Dr,JLM-qaff-RTT-Dr-BD,JLM-qaff-RTT-Dr-C}),
Ko\v{z}i\'{c} constructed the quantum vertex operator algebras associated with trigonometric
$R$-matrices of types $A$, $B$, $C$ and $D$ (\cite{Kozic-qva-tri-A, K-qva-phi-mod-BCD}),
and established a one-to-one correspondence between $\phi$-coordinated modules and
restricted modules for quantum affine algebras.
In \cite{JKLT-Defom-va}, we developed a method for constructing quantum vertex operator algebras by using vertex bialgebras.
Applying this method together with the Drinfeld's quantum affinization construction (\cite{Dr-new,J-KM,Naka-quiver}), we constructed in \cite{K-Quantum-aff-va} the quantum affine vertex algebras $V_{\hat\g,\hbar}(\ell,0)$ and $L_{\hat\g,\hbar}(\ell,0)$ for all symmetric Kac-Moody Lie algebras $\g$,
and established an isomorphism between the category of their $\phi$-coordinated modules and the category of certain restricted modules of untwisted quantum affinization of $\g$ of level $\ell$.
When $\g$ is of finite type and $\ell\in\Z_+$, we proved that $L_{\hat\g,\hbar}(\ell,0)/\hbar L_{\hat\g,\hbar}(\ell,0)$
is isomorphic to the simple affine vertex algebra $L_{\hat\g}(\ell,0)$.


Let $A$ be a symmetrizable generalized Cartan matrix and let $\g$ be the corresponding Kac-Moody Lie algebra.
In this paper, we present a unified construction of quantum affine vertex algebras $V_{\hbar,\tau}^\ell(\g)$ for each $\ell\in \C$
and $L_{\hbar, \tau}^\ell(\g)$ for each $\ell\in\Z_+$,
recovering the $\hbar$-adic weak quantum vertex algebra $\mathcal V_A(\ell)$ (see Appendix \ref{app:YD}) and the $\hbar$-adic quantum vertex algebras $V_{\hat\g,\hbar}(\ell,0)$ and $L_{\hat\g,\hbar}(\ell,0)$ (see Appendix \ref{app:qaff}).
The construction depends on a family $\tau$ of power series in $\C[[z,\hbar]]$; the collection of all such families forms an abelian group $(\mathfrak T,\ast,\varepsilon)$.
By choosing a specific $\tau\in\mathfrak T$, we establish an isomorphism between the category of equivariant $\phi$-coordinated quasi-modules of $V_{\hbar,\tau}^\ell(\g)$ and the category of restricted modules of the twisted quantum affinization algebra $\U_\hbar(\wh\g_\mu)$ of level $\ell$ (see Section \ref{sec:tqaff}).

Unlike affine vertex algebras, the quantum affine vertex algebras associated to double Yangians, untwisted quantum affine algebras, and twisted quantum affine algebras are not isomorphic.
However, we prove that these quantum affine vertex algebras contain dense quantum vertex subalgebras equipped with a increasing filtration, whose associated graded algebras coincide.
More precise, both $V_{\hbar,\varepsilon}^\ell(\g)$ and $L_{\hbar,\varepsilon}^\ell(\g)$ can be realized as  $\hbar$-adic quantum vertex quotient algebras of a deformation of a $\Z$-graded vertex algebra by using a vertex bialgebra (see Section \ref{sec:qaffva}).
Using this, we prove that they both contain $\Z$-graded dense quantum vertex subalgebras, which are denoted by $V_\varepsilon^\ell(\g)$ and $L_\varepsilon^\ell(\g)$, respectively (see Section \ref{sec:Z-grading}).
Moreover, for each $\tau\in\mathfrak T$, the $\hbar$-adic quantum vertex algebras $V_{\hbar,\tau}^\ell(\g)$ and $L_{\hbar,\tau}^\ell(\g)$ can also be realized as deformations of $V_{\hbar,\varepsilon}^\ell(\g)$ and $L_{\hbar,\varepsilon}^\ell(\g)$ by using vertex bialgebras (see Section \ref{sec:qaffva}).
Using this, we prove that the dense quantum vertex subalgebras with a increasing filtration obtained from $V_\varepsilon^\ell(\g)$ and $L_\varepsilon^\ell(\g)$ induce dense quantum vertex subalgebras with a increasing filtration in $V_{\hbar,\tau}^\ell(\g)$ and $L_{\hbar,\tau}^\ell(\g)$, whose associated graded algebras are exactly the $\Z$-graded quantum vertex algebra $V_\varepsilon^\ell(\g)$ and $L_\varepsilon^\ell(\g)$ (see Section \ref{sec:filtration}).

This paper is organized as follows: In Section \ref{sec:qva}, we recall the notions of $\hbar$-adic quantum vertex algebras and equivariant $\phi$-coordinated quasi-modules. In Section \ref{sec:deform}, we recall the deformation approach introduced in \cite{JKLT-Defom-va}.
In Section \ref{sec:construct}, we give an $\hbar$-adic counterpart of the construction of quantum vertex algebra introduced in \cite[Section 4.2]{K-qaffva-rtu}.
In Section \ref{sec:qaffva}, we present a unified construction of quantum affine vertex algebras and then realize them as deformations by using vertex bialgebras.
We show that these quantum affine vertex algebras recover the $\hbar$-adic weak quantum vertex algebra associated to double Yangians and untwisted quantum affinization algebras introduced in \cite{KL-YD-1} and \cite{K-Quantum-aff-va} (see Appendices \ref{app:YD} and \ref{app:qaff}).
Moreover, by choosing a specific $\tau\in \mathfrak T$, we establish, in Section \ref{sec:tqaff}, an isomorphism between the category of equivariant $\phi$-coordinated quasi-modules of $V_{\hbar,\tau}^\ell(\g)$ and the category of restricted modules of the twisted quantum affinization algebra $\U_\hbar(\wh\g_\mu)$ of level $\ell$.
In Section \ref{sec:Z-grading}, we prove that both $V_{\hbar,\varepsilon}^\ell(\g)$ and $L_{\hbar,\varepsilon}^\ell(\g)$ contain $\Z$-graded dense quantum vertex subalgebras, denoted  $V_\varepsilon^\ell(\g)$ and $L_\varepsilon^\ell(\g)$, respectively.
In Section \ref{sec:filtration}, we prove that the quantum affine vertex algebras $V_{\hbar,\tau}^\ell(\g)$ and $L_{\hbar,\tau}^\ell(\g)$ contain dense quantum vertex subalgebras equipped with a increasing filtration, and that their associated graded algebras are exactly the $\Z$-graded quantum vertex algebra $V_\varepsilon^\ell(\g)$ and $L_\varepsilon^\ell(\g)$.

Throughout this paper, we denote the set of nonnegative integer and  positive integers by $\N$ and $\Z_+$, respectively.
For any ring $R$, we denote the set of invertible elements by $R^\times$.

\section{Quantum vertex algebra}\label{sec:qva}

A $\C[[\hbar]]$-module $V$
is said to be \emph{torsion-free} if $\hbar v\ne 0$ for every $0\ne v\in V$,
and said to be \emph{separated} if $\cap_{n\ge1}\hbar^n V=0$.
For a $\C[[\hbar]]$-module $V$, using subsets $v+\hbar^nV$ for $v\in V$, $n\ge 1$
as the basis of open subsets one obtains a topology on $V$, which is called the
\emph{$\hbar$-adic topology}.
A $\C[[\hbar]]$-module $V$ is said to be \emph{$\hbar$-adically complete}
if every Cauchy sequence in $V$ with respect to this $\hbar$-adic topology has a limit in $V$.
A $\C[[\hbar]]$-module $V$ is \emph{topologically free} if $V=V_0[[\hbar]]$
for some vector space $V_0$ over $\C$.
It is known that a $\C[[\hbar]]$-module is topologically free if and only if it is torsion-free, separated, and $\hbar$-adically complete (\cite{Kassel-topologically-free}, see \cite{Li-h-adic}).
For each subset $K$ of $V$, we denote by $\overline{K}$ the closure of $K$ and set
\begin{align}\label{eq:Husdorff}
  [K]=\set{v\in V}{\hbar^nv\in K\,\,\te{for some }n\in\N}.
\end{align}
When $K$ is a submodule of $V$, a similar argument of \cite[Proposition 3.7]{Li-h-adic} shows that 
\begin{align*}
  \left[\overline{[K]}\right]=\overline{[K]}.
\end{align*}
For another topologically free $\C[[\hbar]]$-module $U=U_0[[\hbar]]$, we recall the complete tensor
\begin{align*}
  U\wh\ot V=(U_0\ot V_0)[[\hbar]].
\end{align*}

We view a vector space as a $\C[[\hbar]]$-module by letting $\hbar=0$.
Fix a $\C[[\hbar]]$-module $W$. For $k\in\Z_+$,
and some formal variables $z_1,\dots,z_k$, we define
\begin{align}
  \E^{(k)}(W;z_1,\dots,z_k)=\Hom_{\C[[\hbar]]}(W,W((z_1,\dots,z_k))).
\end{align}
For a subset $\Gamma$ of $\C^\times$, we let
$\C_\Gamma[z]$ be the monoid generated by polynomials of the form $z-\al$, where $\al\in \Gamma$.
Then an ordered sequence $(a_1(z),\dots, a_k(z))$ in $\E(W)$ is called \emph{$\Gamma$-compatible} if there exists $f(z_1,z_2)\in \C_\Gamma[z_1/z_2]$
such that
\begin{align*}
  &\prod_{1\le i<j\le k}f(z_i,z_j)a_1(z_1)a_2(z_2)\cdots a_k(z_k)\in\E^{(k)}(W).
\end{align*}
In addition, this sequence is called \emph{compatible} if $\Gamma=\{1\}$, and is called \emph{quasi-compatible}
if $\Gamma=\C^\times$.
Moreover, a subset $U\subset \E(W)$ is called ($\Gamma$-, quasi-)compatible if every ordered sequence is ($\Gamma$-, quasi-)compatible.

Let $\phi(z,z_1)\ne z$ be an associate of the additive formal group law $F_a(z_1,z_2)=z_1+z_2$.
To be more precise, $z\ne\phi(z,z_1)\in \C((z))[[z_1]]$ such that
\begin{align*}
  \phi(z,0)=z,\quad\te{and}\quad \phi(\phi(z,z_1),z_2)=\phi(z,z_1+z_2).
\end{align*}
For a quasi-compatible ordered pair $(a(z),b(z))$ such that
\begin{align*}
  f(z_1,z_2)a(z_1)b(z_2)\in\E^{(2)}(W)\quad\te{for some }0\ne f(z_1,z_2)\in \C[z_1,z_2],
\end{align*}
a vertex operator map $Y_\E^\phi$ was defined in \cite{Li-phi-coor}:
\begin{align}\label{eq:def-Y-E}
  &Y_\E^\phi(a(z),z_0)b(z)=\sum_{n\in\Z}a(z)_n^\phi b(z)z_0^{-n-1}\\
  &\quad\nonumber=\iota_{z,z_0}f(\phi(z,z_0),z)\inv \left.\left( f(z_1,z)a(z_1)b(z) \right)\right|_{z_1=\phi(z,z_0)}.
\end{align}


Now, we assume $W=W_0[[\hbar]]$ for some vector space $W_0$. Then $W$ is topologically free.
Define
\begin{align}\label{eq:E-hbar}
  \E_\hbar^{(k)}(W;z_1,\dots,z_k)=\Hom_{\C[[\hbar]]}(W,W_0((z_1,\dots,z_k))[[\hbar]]).
\end{align}
Note that $\E_\hbar^{(k)}(W;z_1,\dots,z_k)=\E^{(k)}(W;z_1,\dots,z_k)[[\hbar]]$ is topologically free.
For convenience, we will also write $\E_\hbar^{(k)}(W)=\E_\hbar^{(k)}(W;z_1,\dots,z_k)$ and write $\E_\hbar(W)=\E_\hbar^{(1)}(W)$.
For $n,k\in\Z_+$, the quotient map from $W$ to $W/\hbar^nW$ induces the following $\C[[\hbar]]$-module map
\begin{align*}
  \wt\pi_n^{(k)}:(\End_{\C[[\hbar]]} (W))[[z_1^{\pm 1},\dots,z_k^{\pm 1}]]\to (\End_{\C[[\hbar]]}(W/\hbar^nW))[[z_1^{\pm 1},\dots,z_k^{\pm 1}]].
\end{align*}
For $A(z_1,z_2),B(z_1,z_2)\in\Hom_{\C[[\hbar]]}(W,W_0((z_1))((z_2))[[\hbar]])$, we write $A(z_1,z_2)\sim B(z_2,z_1)$ if for each $n\in\Z_+$ there is $k\in\N$ such that
\begin{align*}
  (z_1-z_2)^k\wt\pi_n^{(2)}(A(z_1,z_2))= (z_1-z_2)^k\wt\pi_n^{(2)}(B(z_2,z_1)).
\end{align*}
Let $Z(z_1,z_2):\E_\hbar(W)\wh \ot \E_\hbar(W)\wh\ot\C((z))[[\hbar]]\to\End_{\C[[\hbar]]}(W)[[z_1^{\pm 1},z_2^{\pm 1}]]$ be defined by
\begin{align*}
  Z(z_1,z_2)\left(a(z)\ot b(z)\ot f(z)\right)=\iota_{z_1,z_2}f(z_1-z_2)a(z_1)b(z_2).
\end{align*}
A subset $U$ of $\E_\hbar(W)$ is said to be \emph{$\hbar$-adically $S$-local} if for any $a(z),b(z)\in U$, there exists $A(z)\in \left(\C U\ot \C U\ot \C((z))\right)[[\hbar]]$ such that
\begin{align*}
  a(z_1)b(z_2)\sim Z(z_2,z_1)\left(A(z)\right),
\end{align*}
where $\C U$ denotes the subspace spanned by $U$.

Recall from \cite[Remark 4.7]{Li-h-adic} that $\wt \pi_n^{(k)}$ induces a $\C[[\hbar]]$-module map
$\pi_n^{(k)}:\E_\hbar^{(k)}(W)\to \E^{(k)}(W/\hbar^n W)$ with kernel $\hbar^n\E_\hbar^{(k)}(W)$.
And $\E_\hbar^{(k)}(W)$ is isomorphic to the inverse limit of the following inverse system
\begin{align*}
  \begin{tikzcd}[column sep=1.5em, ampersand replacement=\&]
    0 \& \E^{(k)}(W/\hbar W)\ar[l]\& \E^{(k)}(W/\hbar^2 W)\ar[l]\& \E^{(k)}(W/\hbar^3 W)\ar[l]\&\cdots\ar[l].
  \end{tikzcd}
\end{align*}
If $k=1$, we will also write $\pi_n=\pi_n^{(1)}$.
Then an ordered sequence $(a_1(z)\dots,a_r(z))$ in $\E_\hbar(W)$ is said to be \emph{$\hbar$-adically ($\Gamma$-, quasi-) compatible} if for every $n\in\Z_+$,
the sequence $(\pi_n(a_1(z)),\dots,\pi_n(a_r(z)))$ in $\E(W/\hbar^nW)$ is ($\Gamma$-, quasi-) compatible.
A subset $U$ of $\E_\hbar(W)$ is said to be \emph{$\hbar$-adically ($\Gamma$-, quasi-) compatible} if every finite sequence in $U$ is $\hbar$-adically ($\Gamma$-, quasi-) compatible.

Let $(a(z),b(z))$ in $\E_\hbar(W)$ be $\hbar$-adically quasi-compatible.
We have the following vertex operator map (see \cite[Definition 2.17]{JKLT-G-phi-mod}):
\begin{align}\label{eq:def-Y-E}
  &Y_\E^\phi\left(a(z),z_0\right)b(x)
  =\sum_{n\in\Z}a(z)_nb(z)z_0^{-n-1}\\
  =&\varprojlim_{n>0}Y_\E^\phi(\pi_n(a(z)),z_0)\pi_n(b(z)).\nonumber
\end{align}
The formal group law $F_a$ is obvious an associate of $F_a$. We denote the corresponding vertex operator map $Y_\E^{F_a}$ by $Y_\E$ for short.

An \emph{$\hbar$-adic nonlocal vertex algebra} \cite[Definition 2.9]{Li-h-adic}
is a topologically free $\C[[\hbar]]$-module $V$ equipped with a $\C[[\hbar]]$-module map
\begin{align*}
  Y(\cdot,z):V&\longrightarrow \E_\hbar(V),\quad
  v\mapsto Y(v,z)=\sum_{n\in\Z}v_nz^{-n-1},
\end{align*}
and a distinguished vacuum vector $\vac$ such that the following vacuum property holds
\begin{align}\label{eq:vacuum-property}
  Y(\vac,z)v=v,\quad Y(v,z)\vac\in V[[z]],\quad \lim_{z\to 0}Y(v,z)\vac=v\quad\te{for }v\in V,
\end{align}
and for $u,v\in V$, $(Y(u,z),Y(v,z))$ is an $\hbar$-adic compatible pair with
\begin{align}
  &Y_\E\left(Y(u,z),z_0\right)Y(v,z)=Y\left(Y(u,z_0)v,z\right)\label{eq:weak-asso}
\end{align}
Denote by $\partial$ the canonical derivation on $V$:
\begin{align}\label{eq:canonical-der}
  u\mapsto\partial u=\lim_{z\to 0}\frac{d}{dz}Y(u,z)\vac.
\end{align}

An \emph{$\hbar$-adic weak quantum vertex algebra} (\cite[Definition 2.9]{Li-h-adic}) is an $\hbar$-adic nonlocal vertex algebra $V$, such that $\set{Y(u,z)}{u\in V}$ is $\hbar$-adically $S$-local.
And an \emph{$\hbar$-adic quantum vertex algebra} (\cite[Definition 2.20]{Li-h-adic}) is an $\hbar$-adic weak quantum vertex algebra $V$ equipped with a $\C[[\hbar]]$-module map (called a \emph{quantum Yang-Baxter operator})
\begin{align*}
  S(z):V\wh\ot V&\longrightarrow V\wh\ot V\wh\ot \C((z))[[\hbar]]\\
  v\ot u&\mapsto \sum_{n\in\Z}S_n(v\ot u)z^{-n-1},
\end{align*}
which satisfies the \emph{shift condition}:
\begin{align}\label{eq:qyb-shift}
  [\partial\ot 1,S(z)]=-\frac{d}{dz}S(z),
\end{align}
the \emph{quantum Yang-Baxter equation}:
\begin{align}\label{eq:qyb}
  S^{12}(z_1)S^{13}(z_1+z_2)S^{23}(z_2)=S^{23}(z_2)S^{13}(z_1+z_2)S^{12}(z_1),
\end{align}
and the \emph{unitarity condition}:
\begin{align}\label{eq:qyb-unitary}
  S^{21}(z)S(-z)=1,
\end{align}
satisfying the following conditions:

  (1) The \emph{vacuum property}:
  \begin{align}\label{eq:qyb-vac-id}
    S(z)(\vac\ot v)=\vac\ot v,\quad \te{for }v\in V.
  \end{align}

 (2) The \emph{$S$-locality}:
  For any $u,v\in V$, one has
  \begin{align}\label{eq:qyb-locality}
  Y(u,z_1)Y(v,z_2)\sim Y(z_2)(1\ot Y(z_1))S(z_2-z_1)(v\ot u).
  \end{align}

  (3) The \emph{hexagon identity}:
  \begin{align}\label{eq:qyb-hex-id}
    S(z_1)(Y(z_2)\ot 1)=(Y(z_2)\ot 1)S^{23}(z_1)S^{13}(z_1+z_2).
  \end{align}

\section{Deformation by vertex bialgebras}\label{sec:deform}

In this section, we recall the deformation method by vertex bialgebras developed in \cite{JKLT-Defom-va}.

\begin{de}
An {\em $\hbar$-adic nonlocal vertex bialgebra} is an $\hbar$-adic nonlocal vertex algebra $H$ with an $\hbar$-adic
coalgebra structure $(\Delta,\epsilon)$ on $H$ such that both $\Delta: H\rightarrow H\wh\otimes H$ and
$\epsilon: H\rightarrow \C[[\hbar]]$ are homomorphisms of $\hbar$-adic nonlocal vertex algebras.
\end{de}

\begin{rem}\label{rem:bialg-der}
Let $(H,\Delta,\varepsilon)$ be an $\hbar$-adic bialgebra equipped with a derivation $\partial$.
Suppose further that $H$ is topologically-free.
Then $H$ is an $\hbar$-adic nonlocal vertex bialgebra with vacuum $1$ and vertex operator map defined by
\begin{align*}
  Y(a,z)b=\left(e^{z\partial}a\right)b\quad\te{for } a,b\in H.
\end{align*}
We denote this $\hbar$-adic nonlocal vertex bialgebra by $(H,\partial,\Delta,\varepsilon)$.
\end{rem}

\begin{de}
Let $(H,\Delta,\varepsilon)$ be an $\hbar$-adic nonlocal vertex bialgebra.
A (left) \emph{$H$-module $\hbar$-adic (nonlocal) vertex algebra} (\cite{Li-smash}) is an $\hbar$-adic nonlocal vertex algebra $V$ equipped with a module structure $\tau(\cdot,z)$ on $V$ for $H$ viewed as an $\hbar$-adic nonlocal vertex algebra such that
\begin{align}
  &\tau(h,z)v\in V\wh\ot \C((z))[[\hbar]],\qquad
  \tau(h,z)\vac=\varepsilon(h)\vac,
  \label{eq:mod-va-for-vertex-bialg1-2}\\
  &
  \tau(h,z_1)Y(u,z_2)v=\sum Y(\tau(h_{(1)},z_1-z_2)u,z_2)\tau(h_{(2)},z_1)v
  \label{eq:mod-va-for-vertex-bialg3}
\end{align}
for $h\in H$, $u,v\in V$, where $\Delta(h)=\sum h_{(1)}\ot h_{(2)}$ is the coproduct in the Sweedler notation.
\end{de}

The following result was introduced in \cite[Theorem 4.9]{Li-smash}.

\begin{thm}
Let $(H,\Delta,\varepsilon)$ be an $\hbar$-adic nonlocal vertex bialgebra and let $(V,\tau)$ be an $H$-module $\hbar$-adic nonlocal vertex algebra.
Set $V\sharp H=V\wh\ot H$ as a topologically-free $\C[[\hbar]]$-module.
For $u,v\in V$, $h,h'\in H$ define
\begin{align*}
  Y^\sharp (u\ot h,z)(v\ot h')=\sum Y(u,z)\tau(h_{(1)},z)v\ot Y(h_{(2)},z)h',
\end{align*}
where $\Delta(h)=\sum h_{(1)}\ot h_{(2)}$.
Then $(V\sharp H,Y^\sharp, \vac\ot \vac)$ carries the structure of a nonlocal vertex algebra, which contains $V$ and $H$ canonically as $\hbar$-adic nonlocal vertex subalgebras such that
\begin{align*}
  Y^\sharp(h,z_1)Y^\sharp(u,z_2)=\sum Y^\sharp (\tau(h_{(1)},z_1-z_2)u,z_2)Y^\sharp(h_{(2)},z_1)\quad\te{for }u\in V,\,h\in H.
\end{align*}
\end{thm}

Recall the following notions introduced in \cite{JKLT-Defom-va}.

\begin{de}
Let $(H,\Delta,\varepsilon)$ be an $\hbar$-adic nonlocal vertex bialgebra.
A \emph{(left) $H$-comodule $\hbar$-adic nonlocal vertex algebra} (\cite{JKLT-Defom-va}) is an $\hbar$-adic nonlocal vertex algebra $V$ equipped with a homomorphism
$\rho:V\to V\wh\ot H$ of $\hbar$-adic nonlocal vertex algebras such that
\begin{align}
  (\rho\ot 1)\rho=\sigma^{23}(1\ot \Delta)\rho,\quad (1\ot \varepsilon)\rho=\te{Id}_V.
\end{align}
\end{de}

\begin{de}\label{de:deform-triple}
Let $V$ be an $\hbar$-adic nonlocal vertex algebra.
A \emph{deforming triple} for $V$ is a triple $(H,\rho,\tau)$,
where $H$ is an $\hbar$-adic nonlocal vertex bialgebra,
$(V,\rho)$ is a left $H$-comodule $\hbar$-adic nonlocal vertex algebra and $(V,\tau)$ is an $H$-module $\hbar$-adic nonlocal vertex algebra, such that
\begin{align}\label{eq:rho-tau-compatible}
  \rho(\tau(h,z)v)=(\tau(h,z)\ot 1)\rho(v)\quad \te{for }h\in H,\,v\in V.
\end{align}
\end{de}

In the rest of this section, we fix a cocommutative $\hbar$-adic nonlocal vertex bialgebra $H$ and an $H$-comodule $\hbar$-adic nonlocal vertex algebra $(V,\rho)$.
We note that $$\Hom(H,\Hom(V,V\wh\ot\C((z))[[\hbar]]))$$ is a unital associative algebra,
where the multiplication is defined by
\begin{align*}
  (f\ast g)(h,z)u=\sum f(h_{(1)},z)g(h_{(2)},z)u
\end{align*}
for $f,g\in\Hom(H,\Hom(V,V\wh\ot\C((z))[[\hbar]]))$,
and the unit $\varepsilon$ defined by
\begin{align*}
  \varepsilon(h,z)u=\varepsilon(h)u\quad\te{for }h\in H,\,u\in V.
\end{align*}
For $f,g\in\Hom(H,\Hom(V,V\wh\ot\C((z))[[\hbar]]))$, we say that $f$ and $g$ commute if
\begin{align*}
  [f(h,z_1),g(k,z_2)]=0\quad \te{for }h,k\in H.
\end{align*}
We say that $f$ is invertible if $f\ast g=\varepsilon$.

\begin{prop}\label{prop:deform}
Define
\begin{align*}
  T_\tau(z):V\wh\ot V&\longrightarrow V\wh\ot V\wh\ot \C((z))[[\hbar]],\\
  u\ot v&\mapsto \sum \tau(v_{(2)},-z)u\ot v_{(1)}
\end{align*}
where $\rho(v)=\sum v_{(1)}\ot v_{(2)}\ot V\wh\ot H$,
and define
$\mathfrak D_\tau^\rho(Y)(\cdot,z):V\to\E_\hbar(V)$ by
\begin{align*}
  \mathfrak D_\tau^\rho(Y)(z)=Y(z)T_\tau^{21}(-z).
\end{align*}
Then $(V,\mathfrak D_\tau^\rho(V),\vac)$ is an $\hbar$-adic nonlocal vertex algebra, which is denoted by $\mathfrak D_\tau^\rho(V)$.
Moreover, $(\mathfrak D_\tau^\rho(V),\rho)$ is an $H$-comodule $\hbar$-adic nonlocal vertex algebra.
Furthermore, $\rho:\mathfrak D_\tau^\rho(V)\to V\sharp H$ is an $\hbar$-adic nonlocal vertex algebra homomorphism.
\end{prop}

Denote by $\mathfrak L_H^\rho(V)$ be the set of all $H$-module $\hbar$-adic nonlocal vertex algebra structures $\tau$ such that $(H,\rho,\tau)$ becomes a deforming triple of $V$.

\begin{rem}\label{rem:trivial-deform}
The counit $\varepsilon$ defines an $H$-module $\hbar$-adic nonlocal vertex algebra structure on $V$ defined by
\begin{align*}
  \varepsilon(h,z)v=\varepsilon(h)v\quad\te{for }h\in H,\,v\in V.
\end{align*}
Then
$\varepsilon\in \mathfrak L_H^\rho(V)$ and
$\mathfrak D_\varepsilon^\rho(V)=V$.
\end{rem}

\begin{prop}\label{prop:deform-mult}
Let $\tau$ and $\tau'$ be commuting elements in $\mathfrak L_H^\rho(V)$. Then
$\tau\ast\tau'\in \mathfrak L_H^\rho(V)$ and $\tau\ast\tau'=\tau'\ast\tau$. Moreover, $\tau\in\mathfrak L_H^\rho(\mathfrak D_{\tau'}^\rho(V))$.
Furthermore,
\begin{align*}
  \mathfrak D_\tau^\rho(\mathfrak D_{\tau'}^\rho(V))
  =\mathfrak D_{\tau\ast\tau'}^\rho(V).
\end{align*}
\end{prop}

The following is the $\hbar$-adic analogue of \cite[Theorem 2.49]{K-qaffva-rtu}.

\begin{thm}\label{thm:deform-qva}
Let $\tau\in\mathfrak L_H^\rho(V)$ such that $\tau$ is invertible and commuting with itself.
Suppose that $V$ is an $\hbar$-adic quantum vertex algebra with quantum Yang-Baxter operator $S(z)$,
and suppose further that
\begin{align*}
  &(\rho\ot 1)S(z)=S^{13}(z)(\rho\ot 1),\quad
  (1\ot\rho)S(z)=S^{12}(z)(1\ot\rho),\\
  &(\tau(h,z_1)\ot 1)S(z_2)=S(z_2)(\tau(h,z_1)\ot 1),\\
  &(1\ot \tau(h,z_1))S(z_2)=S(z_2)(1\ot\tau(h,z_1))
  \quad\te{for }h\in H.
\end{align*}
Then $\mathfrak D_\tau^\rho(V)$ is an $\hbar$-adic quantum vertex algebra with quantum Yang-Baxter operator $S_\tau(z)$ defined by
\begin{align*}
  S_\tau(z)=T_\tau^{21}(-z)\inv S(z)T_\tau(z).
\end{align*}
\end{thm}

The following result is an immediate consequence of Theorem \ref{thm:deform-qva}, which was given in \cite[Theorom 3.8]{JKLT-G-phi-mod}.

\begin{coro}\label{coro:deform-qva}
Let $\tau\in\mathfrak L_H^\rho(V)$ such that $\tau$ is invertible and commuting with itself.
Suppose that $V$ is an $\hbar$-adic vertex algebra.
Then $\mathfrak D_\tau^\rho(V)$ is an $\hbar$-adic quantum vertex algebra with quantum Yang-Baxter operator $S_\tau(z)$ defined by
\begin{align*}
  S_\tau(z)(v\ot u)=
  \sum \tau(u_{(2)},-z)v_{(1)}\ot \tau\inv(v_{(2)},z)u_{(1)}\quad
  \te{for }u,v\in V.
\end{align*}
\end{coro}

\section{General construction}\label{sec:construct}

In this section, we give an $\hbar$-adic counterpart of the construction of quantum vertex algebra introduced in \cite[Section 4.2]{K-qaffva-rtu}.
Let $J^0$ and $J^\pm$ be some fixed countable sets.
We note that the set $\mathfrak V$ of tuples
\begin{align*}
  &\vartheta=(\vartheta_{ij}^{st}(z)\,|\,i\in J^s,\,j\in J^t\,s,t\in\{0,\pm\})\in \prod_{s,t\in\{0,\pm\}}\C((z))^{J_s\times J_t}
\end{align*}
such that $\vartheta_{ij}^{st}(z)\ne 0$ if $s,t\in\{\pm\}$,
carries an abelian group structure with multiplication
$\vartheta\ast\bar\vartheta$ defined by
\begin{align*}
  (\vartheta\ast\bar\vartheta)_{ij}^{st}(z)=
  \begin{cases}
    \vartheta_{ij}^{st}(z)+\bar\vartheta_{ij}^{st}(z), & \mbox{if } 0\in\{s,t\},\\
    \vartheta_{ij}^{st}(z)\bar\vartheta_{ij}^{st}(z), & \mbox{otherwise},
  \end{cases}
\end{align*}
the identity $\varepsilon$ defined by
\begin{align*}
  \varepsilon_{ij}^{st}(z)=
  \begin{cases}
    0, & \mbox{if }0\in\{s,t\},\\
    1, & \mbox{otherwise},
  \end{cases}
\end{align*}
and the inverse $\vartheta\inv$ defined by
\begin{align*}
  (\vartheta\inv)_{ij}^{st}(z)=
  \begin{cases}
    -\vartheta_{ij}^{st}(z), & \mbox{if }0\in \{s,t\},\\
    \vartheta_{ij}^{st}(z)\inv, & \mbox{otherwise}.
  \end{cases}
\end{align*}

\begin{de}
For $\vartheta\in \mathfrak V$, we denote by
$\mathcal M_\vartheta$ the category, whose objects are topologically-free $\C[[\hbar]]$-modules $W$ equipped with fields
\begin{align*}
  \set{a_i^b(z)}{i\in J^b,\,\,b\in\{0,\pm\}}\subset\E_\hbar(W),
\end{align*}
satisfying the relations below
\begin{align}
  &[a_i^0(z_1),a_j^0(z_2)]
  =\vartheta_{ij}^{0,0}(z_1-z_2)-\vartheta_{ji}^{0,0}(z_2-z_1),
  \label{eq:vartheta1}\\
  &[a_i^0(z_1),a_j^\pm(z_2)]=\pm a_j^\pm(z_2)\left(\vartheta_{ij}^{0,\pm}(z_1-z_2)
    +\vartheta_{ji}^{\pm,0}(z_2-z_1)\right),\label{eq:vartheta2}\\
  &\vartheta_{ij}^{\epsilon_1,\epsilon_2}(z_1-z_2)
  a_i^{\epsilon_1}(z_1)a_j^{\epsilon_2}(z_2)
  =
  \vartheta_{ji}^{\epsilon_2,\epsilon_1}(z_2-z_1)
  a_j^{\epsilon_2}(z_1)a_i^{\epsilon_1}(z_1).\label{eq:vartheta3}
\end{align}
\end{de}

\begin{prop}\label{prop:M-qva}
There exists
\begin{align*}
  (V(\mathcal M_\vartheta),a_i^b(z))\in \obj\mathcal M_\vartheta,\quad \vac\in V(\mathcal M_\vartheta),\quad\te{and}\quad
  a_i^b\in V,\,\,i\in J^b,\,\,b\in\{0,\pm\},
\end{align*}
such that $V(\mathcal M_\vartheta)$ carries the structure of an $\hbar$-adic quantum vertex algebra with vacuum vector $\vac$, and the vertex operator map $Y$ uniquely determined by
\begin{align*}
  Y(a_i^b,z)=a_i^b(z)\quad\te{for }b\in\{0,\pm\},\,i\in J^b,
\end{align*}
and the quantum Yang-Baxter operator
$S_\vartheta(z)$ determined by
\begin{align*}
  &S_\vartheta(z)(a_j^0\ot a_i^0)=a_j^0\ot a_i^0
  +\vac\ot\vac\ot (\vartheta_{ij}^{0,0}(-z)-\vartheta_{ji}^{0,0}(z)),\\
  &S_\vartheta(z)(a_j^0\ot a_i^\pm)=a_j^0\ot a_i^\pm\mp\vac\ot a_i^\pm\ot(\vartheta_{ij}^{\pm,0}(-z)
    +\vartheta_{ji}^{0,\pm}(z)),\\
  &S_\vartheta(z)(a_j^{\pm}\ot a_i^0)=a_j^\pm\ot a_i^0
  \pm a_j^\pm\ot\vac\ot( \vartheta_{ij}^{0,\pm}(-z)+\vartheta_{ji}^{\pm,0}(z) ),\\
  &S_\vartheta(z)(a_j^{\epsilon_1}\ot a_i^{\epsilon_2})
  =a_j^{\epsilon_1}\ot a_i^{\epsilon_2}\ot
  \vartheta_{ij}^{\epsilon_2,\epsilon_1}(-z)\inv \vartheta_{ji}^{\epsilon_1,\epsilon_2}(z).
\end{align*}
Moreover, for each $V(\mathcal M_\vartheta)$-module $(W,Y_W)$,
\begin{align*}
  (W,Y_W(a_i^b,z))\in \obj\mathcal M_\vartheta.
\end{align*}
Furthermore, for each $(W,a_i^b(z))\in \obj\mathcal M_\vartheta$,
there exists a $V(\mathcal M_\vartheta)$-module structure $Y_W$ on $W$ determined by
\begin{align*}
  Y_W(a_i^b,z)=a_i^b(z)\quad\te{for }i\in J^b,\,\,b\in\{0,\pm\}.
\end{align*}
\end{prop}

The following is an $\hbar$-adic analogue of \cite[Proposition 4.4]{K-qaffva-rtu}.
\begin{prop}\label{prop:universal}
Let $(\bar V,Y,\vac)$ be an $\hbar$-adic nonlocal vertex algebra containing elements $\bar a_i^b\in \bar V$ for  $i\in J^b$, $b\in \{0,\pm\}$, such that
\begin{align*}
  (\bar V,Y(\bar a_i^b,z))\in\obj\mathcal M_\vartheta.
\end{align*}
Then there exists a unique $\hbar$-adic nonlocal vertex algebra homomorphism $V(\mathcal M_\vartheta)\to \bar V$ such that $a_i^b\mapsto \bar a_i^b$.
\end{prop}

\begin{rem}\label{rem:direct-construct}
Suppose that $\vartheta_{ij}^{s,t}(z)\in \C[[\hbar]][z]$ for any $i\in J^s$, $j\in J^t$, $s,t\in\{0,\pm\}$.
The following is a direct construction of $V(\mathcal M_\vartheta)$, which is a straightforward but simpler analogue of the construction given in \cite[Section 4.1]{K-qaffva-rtu}.
Let $A_0$ be the unital free algebra over $\C$ generated by
\begin{align*}
  \set{\wt a_i^s(n)}{i\in J^s,\,s\in\{0,\pm\},\,n\in\Z},
\end{align*}
and let $A=A_0[[\hbar]]$.
Denote by $K_1$ the ideal of $A$ generated by the relations \eqref{eq:vartheta1}, \eqref{eq:vartheta2} and \eqref{eq:vartheta3}.
Define $A'$ to be the quotient algebra of $A$ module the ideal $\overline{[K]}$.
For each $i\in J^s$, $s\in\{0,\pm\}$ and $n\in\Z$, we set
\begin{align*}
  \bar a_i^s(n)=\wt a_i^s(n)+\overline{[K]}\in A.
\end{align*}
Let $A'_+$ be the left ideal of $A'$ generated by
\begin{align*}
  \set{\bar a_i^s(n)}{i\in J^s,\,s\in\{0,\pm\},\,n\in\N}.
\end{align*}
Set
\begin{align*}
  V=A'\Big/\overline{[A'_+]},
\end{align*}
and set
\begin{align*}
  \vac=1+\overline{[A'_+]}\in V,\quad \check a_i^s=\bar a_i^s(-1)\vac\in V\quad\te{for }i\in J^s,\,s\in\{0,\pm\}.
\end{align*}
Then $V$ carries an $\hbar$-adic nonlocal vertex algebra structure with vacuum $\vac$ and vertex operator map $Y$ determined by
\begin{align*}
  Y(\check a_i^s,z)=\sum_{n\in\Z}\bar a_i^s(n)z^{-n-1}
  \quad\te{for }i\in J^s,\,s\in\{0,\pm\}.
\end{align*}
Moreover, there is an $\hbar$-adic nonlocal vertex algebra isomorphism from $V(\mathcal M_\vartheta)$ to $V$ determined by
\begin{align*}
  a_i^s\mapsto\check a_i^s\quad\te{for }i\in J^s,\,s\in\{0,\pm\}.
\end{align*}
\end{rem}

Let $H'$ be the symmetric algebra of the following vector space:
\begin{align*}
  \bigoplus_{b\in\{0,\pm\}}\bigoplus_{i\in J^a}\bigoplus_{n\in\N}\C\partial^n\wh a_i^b.
\end{align*}
Then $H'$ is a commutative and cocommutative bialgebra with $\Delta$ and $\varepsilon$ uniquely determined by
($i\in J$, $n\in\N$):
\begin{align*}
  &\Delta(\partial^n \wh a_i^0)=\partial^n \wh a_i^0\ot 1+1\ot \partial^n \wh a_i^0,\quad \varepsilon(\partial^n \wh a_i^0)=0,\\
  &\Delta(\partial^n \wh a_i^{\pm})=\sum_{k=0}^n\binom{n}{k} \partial^k \wh a_i^{\pm}\ot \partial^{n-k} \wh a_i^{\pm},\quad
  \varepsilon(\partial^n \wh a_i^{\pm})=\delta_{n,0}.
\end{align*}
Let $\partial$ be the derivation on $H'$ such that
\begin{align*}
  \partial (\partial^n \wh a_i^b)=\partial^{n+1}\wh a_i^b\quad\te{for }n\in\N,\,
  i\in J^a,\,b\in\{0,\pm\}.
\end{align*}
It is straightforward to see that $\Delta\circ\partial=(\partial\ot 1+1\ot\partial)\circ\Delta$ and $\varepsilon\circ \partial=0$.
From Remark \ref{rem:bialg-der} we have that $(H',\partial,\Delta,\varepsilon)$ carries a vertex bialgebra structure.
Let $H=H'[[\hbar]]$ be the natural $\hbar$-adic vertex bialgebra.

\begin{prop}\label{prop:deform-iso}
Let $\vartheta,\bar\vartheta\in\mathfrak V$.
There exists a deforming triple $(H,\rho,\bar\vartheta)$ of $V(\mathcal M_\vartheta)$, where $\rho$ is the $H$-comodule $\hbar$-adic nonlocal vertex algebra structure determined by
\begin{align*}
  \rho(a_i^0)=a_i^0\ot 1+\vac\ot \wh a_i^0,\quad
  \rho(a_i^\pm)=a_i^\pm\ot \wh a_i^\pm,
\end{align*}
and $\bar\vartheta\in\mathfrak L_H^\rho(V(\mathcal M_\vartheta))$ determined by
\begin{align*}
  &\bar\vartheta(\wh a_i^0,z)a_j^0=\vac\ot \bar\vartheta_{ij}^{0,0}(z),\quad
  \bar\vartheta(\wh a_i^0,z)a_j^{\pm}=\pm a_j^{\pm}\ot \bar\vartheta_{ij}^{0,\pm}(z),
  \\
  &\bar\vartheta(\wh a_i^\pm,z)a_j^0=a_j^0\ot 1\mp\vac\ot \bar\vartheta_{ij}^{\pm,0}(z),\quad
  \bar\vartheta(\wh a_i^\pm z)a_j^{\epsilon}=a_j^{\epsilon}\ot \bar\vartheta_{ij}^{\pm,\epsilon}(z)\inv.
\end{align*}
Moreover,
\begin{align*}
  \mathfrak D_{\bar\vartheta}^\rho(V(\mathcal M_\vartheta))\cong
  V(\mathcal M_{\vartheta\ast\bar\vartheta}).
\end{align*}
\end{prop}

\section{Quantum affine vertex algebra}\label{sec:qaffva}

In this section, we present a unified construction of quantum affine vertex algebras, which recovers both the $\hbar$-adic weak quantum vertex algebras associated to double Yangians (introduced in \cite{KL-YD-1}) and the $\hbar$-adic quantum vertex algebra associated to untwisted quantum affinization algebras (introduced in \cite{K-Quantum-aff-va}).

Let $A=(a_{ij})_{i,j\in I}$ be a symmetrizable generalized Cartan matrix (GCM).
Then there are unique relatively prime positive integers $r_i$ ($i\in I$) such that $DA$ is symmetric with $D=\te{diag}\set{r_i}{i\in I}$.
Let $r$ be the least common multiple of $\set{r_i}{i\in I}$.
Let $\g=[\g(A),\g(A)]$ be the derived subalgebra of the Kac-Moody Lie algebra associated to $A$.
We introduce the $\hbar$-adic vertex algebras.

\begin{de}
Let $\ell\in \C$.
Define $\mathcal M^\ell(\g)$ to be the category, whose objects are topologically free $\C[[\hbar]]$-modules $W$ equipped with fields
\begin{align*}
  h_i(z)=\sum_{n\in\Z}h_i(n)z^{-n-1},\quad x_i^\pm(z)=\sum_{n\in\Z}x_i^\pm(n)z^{-n-1}\in \E_\hbar(W)\quad\te{for }i\in I,
\end{align*}
satisfying the relations 
below
\begin{align}
  \tag{L1}\label{L1}\quad& [h_i(z_1),h_j(z_2)]=r_ia_{ij}r\ell\pd{z_2}z_1\inv\delta\left(\frac{z_2}{z_1}\right),\\
  \tag{L2}\label{L2}\quad& [h_i(z_1),x_j^\pm(z_2)]=\pm r_ia_{ij}x_j^\pm(z_2) z_1\inv\delta\left(\frac{z_2}{z_1}\right),\\
  \label{L3.5}\tag{L3.5}\quad &(z_1-z_2)^{2\delta_{ij}}[x_i^+(z_1),x_j^-(z_2)]=0,\\
  \tag{L4}\label{L4}\quad& (z_1-z_2)[x_i^\pm(z_1),x_j^\pm(z_2)]=0.
\end{align}
Let $R_1^\ell(\g)$ be the ideal of $V(\mathcal M^\ell(\g))$ generated by 
\begin{align*}
  &(x_i^+)_0(x_j^-)-\frac{\delta_{ij}}{r_i}h_i,\quad (x_i^+)_1(x_j^-)-\frac{\delta_{ij}}{r_i}r\ell\vac\quad\te{for }i,j\in I,\\ &\left((x_i^\pm)_0\right)^{m_{ij}}(x_j^\pm)\quad \te{for }i,j\in I\,\te{with }a_{ij}\le 0,\,
  \te{where }m_{ij}=1-a_{ij}.
\end{align*}
Define
\begin{align*}
  V^\ell(\g)=V(\mathcal M^\ell(\g))\Big/\overline{[R_1^\ell(\g)]}.
\end{align*}
If $\ell\in\Z_+$, let $R_2^\ell(\g)$ be the ideal of $V^\ell(\g)$ generated by
\begin{align*}
  (x_i^\pm)_{-1}^{r\ell/r_i}x_i^\pm\quad\te{for }i\in I.
\end{align*}
Define
\begin{align*}
  L^\ell(\g)=V^\ell(\g)\Big/\overline{[R_2^\ell(\g)]}.
\end{align*}
\end{de}

\begin{rem}\label{rem:module-realization}
Let $\wt \g$ be the Lie algebra over $\C$ generated by
$\set{h_i(n),x_i^\pm(n)}{i\in I,n\in\Z}$ subject to relations \eqref{L1}, \eqref{L2}, \eqref{L3.5} and \eqref{L4}.
Then $V(\mathcal M^\ell(\g))$ carries a $\wt \g$-module structure with
\begin{align*}
  h_i(z)=Y(h_i,z),\quad x_i^\pm(z)=Y(x_i^\pm,z)\quad\te{for }i\in I.
\end{align*}
Moreover, let $R$ be the left ideal of $U(\wt \g)[[\hbar]]$ generated by 
\begin{align*}
  h_i(n),\,x_i^\pm(n)\quad\te{for }i\in I,\,n\in\N.
\end{align*}
Then, as a $\wt \g$-module, $V(\mathcal M^\ell(\g))\cong \U(\wt\g)[[\hbar]]/\overline{[R]}$.
\end{rem}

Define $q=\exp\hbar\in \C[[\hbar]]$ and define $q_i=q^{r_i}$.
For $n\in\Z$, set
\begin{align*}
  [n]_x=\frac{x^n-x^{-n}}{x-x\inv}\in\Z[x,x\inv].
\end{align*}
We also define $[n]_v$ for any invertible element $v$ of an algebra over $\C$ with $v^2\ne 1$.
Furthermore, for $k,s\in\N$ such that $0\le k\le s$, define
\begin{align*}
  [s]_v!=[s]_v[s-1]_v\cdots [1]_v,\quad
  \qb{s}{k}_v=\frac{[s]_v!}{[s-k]_v![k]_v!}.
\end{align*}
Introduce the following series
\begin{align}\label{eq:def-op-F}
F(z)=(q^z-q^{-z})/z\in \hbar z^2\C[z^2][[\hbar]].
\end{align}
Let $\mathfrak T$ be the set of tuples
\begin{align*}
  \tau=\big(\tau_{ij}^{a,b}(z)\big)_{i,j\in I}^{a,b\in\{0,\bullet\}}\in \C[[z,\hbar]]^{4|I|^2},
\end{align*}
such that
$\lim\limits_{z\to0}\lim\limits_{\hbar\to 0}\tau_{ij}^{\bullet,\bullet}(z)=1$ and for each $i,j\in I$,
\begin{align*}
  &F\left(\pd z\right)\tau_{ij}^{0,0}(z)=\tau_{ij}^{0,\bullet}(z-r\ell\hbar)
  -\tau_{ij}^{0,\bullet}(z+r\ell\hbar)
  =\tau_{ij}^{\bullet,0}(z-r\ell\hbar)
  -\tau_{ij}^{\bullet,0}(z+r\ell\hbar),\\
  &\tau_{ij}^{\bullet,\bullet}(z-r\ell\hbar)
  \tau_{ij}^{\bullet,\bullet}(z+r\ell\hbar)\inv
  =\exp\left(F\left(\pd z\right)\tau_{ij}^{\bullet,0}(z)\right)
  =\exp\left(F\left(\pd z\right)\tau_{ij}^{0,\bullet}(z)\right).
\end{align*}
Then $\mathfrak T$ carries an abelian group structure with the identity $\varepsilon$ defined by
\begin{align*}
  \varepsilon_{ij}^{a,b}(z)=
  \begin{cases}
    0, & \mbox{if }0\in\{a,b\},\\
    1, & \mbox{otherwise},
  \end{cases}
\end{align*}
the multiplication $\tau\ast\tau'$ defined by
\begin{align*}
  (\tau\ast\tau')_{ij}^{a,b}(z)=
  \begin{cases}
    \tau_{ij}^{a,b}(z)+\tau_{ij}'^{a,b}(z), & \mbox{if }
    0\in\{a,b\},\\
    \tau_{ij}^{a,b}(z)\tau_{ij}'^{a,b}(z), & \mbox{otherwise},
  \end{cases}
\end{align*}
and the inverse $\tau\inv$ defined by
\begin{align*}
  (\tau\inv)_{ij}^{a,b}(z)=
  \begin{cases}
    -\tau_{ij}^{a,b}(z), & \mbox{if }0\in\{a,b\},\\
    \tau_{ij}^{a,b}(z)\inv, & \mbox{otherwise}.
  \end{cases}
\end{align*}

\begin{de}
Let $\ell\in \C$, and $\tau\in\mathfrak T$.
Define $\mathcal M_{\hbar,\tau}^\ell(\g)$ be the category, whose objects are topologically-free $\C[[\hbar]]$-modules $W$ equipped with fields
\begin{align*}
  h_i(z)=\sum_{n\in\Z}h_i(n)z^{-n-1},\quad x_i^\pm(z)=\sum_{n\in\Z}x_i^\pm(n)z^{-n-1}\in \E_\hbar(W)\quad\te{for }i\in I,
\end{align*}
satisfying the relations below
\begin{align}
  \tag{$\tau$1}\label{tau1}
  &[h_i(z_1),h_j(z_2)]\\
  \nonumber=&[r_ia_{ij}]_{q^{\pd{z_2}}}[r\ell]_{q^{\pd{z_2}}}
  (q^{-r\ell\pd{z_2}}\iota_{z_1,z_2}
  -q^{r\ell\pd{z_2}}\iota_{z_2,z_1})(z_1-z_2)^{-2}\\
  &\nonumber+\tau_{ij}^{0,0}(z_1-z_2)-\tau_{ji}^{0,0}(z_2-z_1),\\
  \tag{$\tau$2}\label{tau2}
  &[h_i(z_1),x_j^\pm(z_2)]\\
  \nonumber=&\pm x_j^\pm(z_2)
  [r_ia_{ij}]_{q^{\pd{z_2}}}(q^{-r\ell\pd{z_2}}\iota_{z_1,z_2}
  -q^{r\ell\pd{z_2}}\iota_{z_2,z_1})(z_1-z_2)\inv\\
  &\nonumber\pm x_j^\pm(z_2)(\tau_{ij}^{0,\bullet}(z_1-z_2)+\tau_{ji}^{\bullet,0}(z_2-z_1)),\\
  \tag{$\tau$3.5}\label{tau3.5}
  &(z_1-z_2)^{\delta_{ij}}(z_1-z_2+2r\ell\hbar)^{\delta_{ij}}
  \bigg(x_i^+(z_1)x_j^-(z_2)\\
  -&\nonumber\iota_{z_2,z_1}
  \frac{
    (z_2-z_1+r_ia_{ij}\hbar)\tau_{ij}^{\bullet,\bullet}(z_1-z_2)}
  {(z_2-z_1-r_ia_{ij}\hbar)\tau_{ji}^{\bullet,\bullet}(z_2-z_1)}
  x_j^-(z_2)x_i^+(z_1)\bigg)=0,\\
  \tag{$\tau$4}\label{tau4}
  &(z_1-z_2-r_ia_{ij}\hbar)\tau_{ij}^{\bullet,\bullet}(z_1-z_2)
  x_i^\pm(z_1)x_j^\pm(z_2)\\
  =&\nonumber
  (z_1-z_2+r_ia_{ij}\hbar)\tau_{ji}^{\bullet,\bullet}(z_2-z_1)
  x_j^\pm(z_2)x_i^\pm(z_1).
\end{align}
Let $R_{\tau,1}^\ell(\g)$ be the ideal of $V(\mathcal M_{\hbar,\tau}^\ell(\g))$ generated by
\begin{align}
  &\left(x_i^+\right)_0x_i^-
  -\frac{1}{2r_i\hbar}
    \left(\vac-E_\tau(h_i)\right)\quad\te{for } i\in I,\label{eq:x+0x-}\\
  &\left(x_i^+\right)_1x_i^-
  +\frac{r\ell}{r_i}
  E_\tau(h_i)\quad\te{for } i\in I,\label{eq:x+1x-}\\
  &\left(x_i^\pm\right)_0^{m_{ij}}x_j^\pm\quad\te{for } i,j\in I\,\te{with }a_{ij}\le 0,\label{eq:serre}
\end{align}
where
\begin{align}\label{eq:def-E}
  E_\tau(h_i)=
  \frac{\tau_{ii}^{\bullet,\bullet}(-2r\ell\hbar)^\half}
  {\tau_{ii}^{\bullet,\bullet}(2r\ell\hbar)^\half}
    \exp\left(\left(-q^{-r\ell\partial} F\left(\partial\right)h_i\right)_{-1}\right)\vac.
\end{align}
Define 
\begin{align}
  V_{\hbar,\tau}^\ell(\g)=V(\mathcal M_{\hbar,\tau}^\ell(\g))\Big/ \overline{[R_{\tau,1}^\ell(\g)]}.
\end{align}
If $\ell\in\Z_+$, let $R_{\tau,2}^\ell(\g)$ be the ideal of $V_{\hbar,\tau}^\ell(\g)$ generated by
\begin{align}\label{eq:def-integrable}
  (x_i^\pm)_{-1}^{r\ell/r_i}x_i^\pm\quad\te{for }i\in I.
\end{align}
Define 
\begin{align}
  L_{\hbar,\tau}^\ell(\g)=V_{\hbar,\tau}^\ell(\g)\Big/\overline{[R_{\tau,2}^\ell(\g)]}.
\end{align}
\end{de}

\begin{rem}
For the notations $L^\ell(\g)$ and $L_{\hbar,\tau}^{\ell}(\g)$. we always assume that $\ell\in\Z_+$.
\end{rem}

\begin{rem}
For specific choices of $\tau$, the quantum affine vertex algebras $V_{\hbar,\tau}^\ell(\g)$ and $L_{\hbar,\tau}^\ell(\g)$ recover those associated with the double Yangians introduced in \cite{KL-YD-1} (see Appendix~\ref{app:YD}), and with the untwisted quantum affinization algebras introduced in \cite{K-Quantum-aff-va} (see Appendix~\ref{app:qaff}).
\end{rem}

\begin{rem}\label{rem:x+x-}
Note that the vertex operator map $Y(\cdot,z)$ on $V(\mathcal M_{\hbar,\tau}^\ell(\g))$ is an $\hbar$-adic nonlocal vertex algebra homomorphism
\begin{align*}
  V(\mathcal M_{\hbar,\tau}^\ell(\g))\to \left\langle\set{Y(v,z)}{v\in V(\mathcal M_{\hbar,\tau}^\ell(\g))}\right\rangle.
\end{align*}
Then
\begin{align*}
  Y(E_\tau(h_i),z)=E_\tau\big(Y(h_i,z)\big)
  =E_\tau\big(h_i(z)\big).
\end{align*}
A similar argument to \cite[Proposition 6.10]{K-Quantum-aff-va} shows that \eqref{tau3.5} together with \eqref{eq:x+0x-} and \eqref{eq:x+1x-}
is equivalent to the following relation
\begin{align}
  \tag{$\tau$3}\label{tau3}
  &x_i^+(z_1)x_j^-(z_2)-\iota_{z_2,z_1}
  \frac{
    (z_2-z_1+r_ia_{ij}\hbar)\tau_{ij}^{\bullet,\bullet}(z_1-z_2)}
  {(z_2-z_1-r_ia_{ij}\hbar)\tau_{ji}^{\bullet,\bullet}(z_2-z_1)}
  x_j^-(z_2)x_i^+(z_1)\\
  &\quad=\frac{\delta_{ij}}{2r_i\hbar}\vac z_1\inv\delta\left(\frac{z_2}{z_1}\right)
  -\frac{\delta_{ij}}{2r_i\hbar}E_\tau\big(h_i(z_2)\big)
  z_1\inv\delta\left(\frac{z_2-2r\ell\hbar}{z_1}\right).\nonumber
\end{align}
\end{rem}

\begin{rem}\label{rem:qyb}
For each $\ell\in \C$ and $\tau\in\mathfrak T$,
we denote by $S_\tau(z)$ the quantum Yang-Baxter operator of $V(\mathcal M_{\hbar,\tau}^\ell(\g))$.
Then $S_\tau(z)$ is determined by the following conditions
\begin{align*}
  &S_\tau(z)(h_j\ot h_i)=h_j\ot h_i\\
  &\quad+\vac\ot\vac\ot
  \left( [r_ia_{ij}]_{q^{\pd z}}[r\ell]_{q^{\pd z}}(q^{-r\ell\pd z}-q^{r\ell\pd z})z^{-2}+\tau_{ij}^{0,0}(-z)
  -\tau_{ji}^{0,0}(z) \right),\\
  &S_\tau(z)(h_j\ot x_i^\pm)=h_j\ot x_i^\pm\\
  &\quad\mp\vac\ot x_i^\pm\ot\left(
    [r_ja_{ji}]_{q^{\pd z}}(q^{r\ell\pd z}-q^{-r\ell\pd z})z\inv
    +\tau_{ij}^{\bullet,0}(-z)+\tau_{ji}^{0,\bullet}(z)
  \right),\\
  &S_\tau(z)(x_j^\pm\ot h_i)=x_j^\pm\ot h_i\\
  &\quad\pm x_j^\pm\ot\vac\ot \left(
    [r_ia_{ij}]_{q^{\pd z}}(q^{r\ell\pd z}-q^{-r\ell\pd z})z\inv +\tau_{ij}^{0,\bullet}(-z)+\tau_{ji}^{\bullet,0}(z)
  \right),\\
  &S_\tau(z)(x_j^{\epsilon_1}\ot x_i^{\epsilon_2})
  =x_j^{\epsilon_1}\ot x_i^{\epsilon_2}
  \ot \frac{\tau_{ji}^{\bullet,\bullet}(z)^{\epsilon_1\epsilon_21}}
  {\tau_{ij}^{\bullet,\bullet}(-z)^{\epsilon_1\epsilon_21}}
  \frac{z-\epsilon_1\epsilon_2 r_ia_{ij}\hbar}
  {z+\epsilon_1\epsilon_2 r_ia_{ij}\hbar}.
\end{align*}
\end{rem}

Similar to the proof of \cite[Theorem 6.6]{K-Quantum-aff-va}, we have the following result.
\begin{prop}\label{prop:qyb}
For each $\ell\in \C$ and $\tau\in\mathfrak T$, the $\hbar$-adic nonlocal vertex algebra $V_{\hbar,\tau}^\ell(\g)$ is an $\hbar$-adic quantum vertex algebra with the quantum Yang-Baxter operator induced from $S_\tau(z)$, which is still denoted by $S_\tau(z)$.
Moreover, suppose further that $\ell\in\Z_+$. Then
$L_{\hbar,\tau}^\ell(\g)$ is also an $\hbar$-adic quantum vertex algebra with the quantum Yang-Baxter operator induced from $S_\tau(z)$, which is still denoted by $S_\tau(z)$.
\end{prop}

The following result is an immediate consequence of
Proposition \ref{prop:universal}.

\begin{prop}\label{prop:universal-qaff}
Let $\tau\in \mathfrak T$, and
let $(\bar V,Y,\vac)$ be an $\hbar$-adic nonlocal vertex algebra containing elements
$\set{\bar h_i, \bar x_i^\pm}{i\in I}$,
such that
\begin{align*}
  (\bar V,Y(\bar h_i,z),Y(\bar x_i^\pm,z))\in \obj\mathcal M_{\hbar,\tau}^\ell(\g),
\end{align*}
and the relations \eqref{eq:x+0x-}, \eqref{eq:x+1x-} and \eqref{eq:serre} hold with $h_i$ (resp. $x_i^\pm$) replaced by $\bar h_i$ (resp. $\bar x_i^\pm$).
Then there is a unique $\hbar$-adic nonlocal vertex algebra homomorphism $\varphi:V_{\hbar,\tau}^\ell(\g)\to \bar V$ determined by
\begin{align*}
  h_i\mapsto \bar h_i,\quad x_i^\pm\mapsto \bar x_i^\pm\quad\te{for }i\in I.
\end{align*}
Moreover, if $\ell\in\Z_+$, and the relation \eqref{eq:def-integrable} holds with $x_i^\pm$ replaced by $\bar x_i^\pm$, then the homomorphism $\varphi$ factor through $L_{\hbar,\tau}^\ell(\g)$.
\end{prop}

Next, we realize $V(\mathcal M_{\hbar,\tau}^\ell(\g))$ as deformations of $V(\mathcal M^\ell(\g))$ by using $\hbar$-adic vertex bialgebras.
Let $\tau\in \mathfrak T$.
From Section \ref{sec:construct}, we get a commutative and cocommutative $\hbar$-adic vertex bialgebra $H$ generated by $\partial^n\wh h_i$ and $\partial^n\wh x_i^\pm$ for $i\in I$, $n\in\N$,
with coproduct $\Delta$ and counit $\varepsilon$ determined by
\begin{align*}
  &\Delta(\partial^n \wh h_i)=\partial^n \wh h_i\ot 1+1\ot \partial^n \wh h_i,\quad \varepsilon(\partial^n \wh h_i)=0,\\
  &\Delta(\partial^n \wh x_i^{\pm})=\sum_{k=0}^n\binom{n}{k} \partial^k \wh x_i^{\pm}\ot \partial^{n-k} \wh x_i^{\pm},\quad
  \varepsilon(\partial^n \wh x_i^{\pm})=\delta_{n,0},
\end{align*}
and derivation $\partial$ determined by
\begin{align*}
  \partial (\partial^n \wh a_i)=\partial^{n+1}\wh a_i\quad\te{for }i\in I,\,n\in\N,\,a\in\{h,x^\pm\}.
\end{align*}
We also have a deforming triple $(H,\rho,\wh\tau)$ for $V(\mathcal M^\ell(\g))$ such that
$\rho$ is an $H$-comodule vertex algebra structure determined by
\begin{align}\label{eq:rho}
  \rho(h_i)=h_i\ot 1+\vac\ot \wh h_i,\quad
  \rho(x_i^\pm)=x_i^\pm\ot \wh x_i^\pm\quad\te{for }i\in I,
\end{align}
and $\wh\tau$ is an $H$-module vertex algebra structure determined by
\begin{align}
  &\wh\tau(\wh h_i,z)h_j=\vac\ot \left([r_ia_{ij}]_{q^{\pd z}}
  [r\ell]_{q^{\pd z}}q^{r\ell\pd z}z^{-2}
  -r_ia_{ij}r\ell z^{-2}+\tau_{ij}^{0,0}(z)\right),\\
  &\wh\tau(\wh h_i,z)x_j^\pm=\pm x_j^\pm\ot\left(
  [r_ia_{ij}]_{q^{\pd z}}q^{r\ell\pd z}z\inv
  -r_ia_{ij}z\inv+\tau_{ij}^{0,\bullet}(z)
  \right),\\
  &\wh\tau(\wh x_i^\pm,z)h_j=h_j\ot 1\mp\vac
  \ot\left(
  [r_ia_{ij}]_{q^{\pd z}}q^{r\ell\pd z}z\inv
  -r_ia_{ij}z\inv +\tau_{ij}^{\bullet,0}(z)
  \right),\\
  &\wh\tau(\wh x_i^\pm,z)x_j^\pm=x_j^\pm
  \ot z\inv (z-r_ia_{ij}\hbar)\tau_{ij}^{\bullet,\bullet}(z),\\
  &\wh\tau(\wh x_i^+,z)x_j^-=x_j^-\ot z^{-\delta_{ij}} (z+2r\ell\hbar)^{\delta_{ij}},\\
  &\wh\tau(\wh x_i^-,z)x_j^+=x_j^+\ot z^{-\delta_{ij}}
  (z-2r\ell\hbar)^{\delta_{ij}}
  \frac{(z-r_ja_{ji}\hbar)\tau_{ji}^{\bullet,\bullet}(-z) }{(z+r_ja_{ji}\hbar)\tau_{ij}^{\bullet,\bullet}(z)}.
\end{align}
Proposition \ref{prop:deform-iso} yields the following result, which was also proved in \cite[Proposition 5.12]{K-Quantum-aff-va}.

\begin{lem}\label{lem:deform-realization}
There is an $\hbar$-adic nonlocal vertex algebra isomorphism $V(\mathcal M_{\hbar,\tau}^\ell(\g))
\to \mathfrak D_{\wh\tau}^\rho(V(\mathcal M^\ell(\g)))$ determined by
\begin{align*}
  h_i\mapsto h_i,\quad x_i^\pm\mapsto x_i^\pm\quad\te{for }i\in I.
\end{align*}
\end{lem}

In the rest of this section, we realize quantum affine vertex algebras as deformations of each other by using vertex bialgebras.
Let $\tau,\bar\tau\in\mathfrak T$.
It is easy to check that \eqref{eq:rho} also define an  $H$-comodule $\hbar$-adic nonlocal vertex algebra structure $\rho$ on $V(\mathcal M_{\hbar,\tau}^\ell(\g))$.
From Section \ref{sec:construct}, we also have an $H$-module $\hbar$-adic nonlocal vertex algebra structure
$\bar\tau$ on $V(\mathcal M_{\hbar,\tau}^\ell(\g))$ making $(H,\rho,\bar\tau)$ a deforming triple for $V(\mathcal M_{\hbar,\tau}^\ell(\g))$ such that
$\bar\tau$ is determined by
\begin{align}
  &\bar\tau(\wh h_i,z)h_j=\vac\ot \bar\tau_{ij}^{0,0}(z),\quad
  \bar\tau(\wh h_i,z)x_j^\pm=\pm x_j^\pm\ot \bar\tau_{ij}^{0,\bullet}(z),\label{eq:sigma-0}\\
  &\bar\tau(\wh x_i^\pm,z)h_j=h_j\ot 1\mp\vac\ot \bar\tau_{ij}^{\bullet,0}(z),\quad
  \bar\tau(\wh x_i^\pm z)x_j^\epsilon=x_j^\epsilon\ot \bar\tau_{ij}^{\bullet,\bullet}(z)^{\mp\epsilon 1}.\label{eq:sigma-a}
\end{align}
Moreover, by using Proposition \ref{prop:deform-iso} again, we get the following $\hbar$-adic nonlocal vertex algebra isomorphism
\begin{align*}
  V(\mathcal M_{\hbar,\tau\ast\bar\tau}^\ell(\g))\to\mathfrak D_{\bar\tau}^\rho(V(\mathcal M_{\hbar,\tau}^\ell(\g)))
\end{align*}
determined by
\begin{align*}
  h_i\mapsto h_i,\quad x_i^+\mapsto \bar\tau_{ii}^{\bullet,\bullet}(0)\inv x_i^+,\quad
  x_i^-\mapsto x_i^-.
\end{align*}

\begin{lem}\label{lem:V-h-tau-deform-triple}
Let $X=V$ or $L$ and let $\tau,\bar\tau\in \mathfrak T$.
Define $\bar H$ to be the quotient algebra of $H$ modulo $\overline{[R]}$, where $R$ is the ideal generated by
\begin{align*}
  \partial^n\left(\wh x_i^+\wh x_i^--1\right),\quad
  \partial^n\left((q^{-2r\ell\partial}\wh x_i^+)\wh x_i^-
  -\exp\left(-q^{-r\ell\partial}F(\partial)\wh h_i\right)\right).
\end{align*}
Then the deforming triple $(H,\rho,\tau)$ of $V(\mathcal M_{\hbar,\tau}^\ell(\g))$ induces a deforming triple
$(\bar H,\rho,\bar\tau)$ of $X_{\hbar,\tau}^\ell(\g)$.
\end{lem}

\begin{proof}
It is straightforward to verify that $\bar H$ is a quotient $\hbar$-adic vertex bialgebra of $H$, and that $\rho$ induces an $\bar H$-comodule $\hbar$-adic nonlocal vertex algebra structure on $X_{\hbar,\tau}^\ell(\g)$. Set
\begin{align}\label{eq:def-wt-h}
\wt h_i=-q^{-r\ell\partial}F(\partial)h_i.
\end{align}
Then one can straightforwardly verify that
\begin{align*}
&\bar\tau(\wh h_i,z)\wt h_j
=-\vac\ot q^{r\ell\pd z}F\left(\pd z\right)\bar\tau_{ij}^{0,0}(z),\\
&\bar\tau(\wh x_i^\pm,z)\wt h_j=
\wt h_j\ot 1\pm \vac\ot q^{r\ell\pd z}F\left(\pd z\right)\bar\tau_{ij}^{\bullet,0}(z).
\end{align*}
It follows that
\begin{align*}
&\bar\tau(\wh h_i,z)E_\tau(h_j)=-E_\tau(h_j)\ot q^{r\ell\pd z}F\left(\pd z\right)\bar\tau_{ij}^{0,0}(z),\\
&\bar\tau(\wh x_i^\pm,z)E_\tau(h_j)=E_\tau(h_j)\ot \exp\left(\pm q^{r\ell\pd z}F\left(\pd z\right)\bar\tau_{ij}^{\bullet,0}(z)\right).
\end{align*}
Using these relations, one can straightforwardly verify that $\bar\tau(\cdot,z)$ induces an $H$-module nonlocal vertex algebra structure on $X_{\hbar,\tau}^\ell(\g)$.
Since $H$ is commutative, we have
\begin{align*}
&\bar\tau(\wh x_i^+\wh x_i^-,z)=\bar\tau(\wh x_i^+,z)\bar\tau(\wh x_i^-,z),\quad
 \bar\tau((q^{-2r\ell\partial}\wh x_i^+)\wh x_i^-,z)
=\bar\tau(\wh x_i^+,z-2r\ell\hbar)
\bar\tau(\wh x_i^-,z).
\end{align*}
Using these, one can straightforwardly verify that $\bar\tau(\cdot,z)$ factors through $\bar H$, as desired.
\end{proof}

\begin{prop}\label{prop:V-h-tau-deform}
Let $X=V$ or $L$.
For $\tau,\bar\tau\in\mathfrak T$, there are $\hbar$-adic quantum vertex algebras isomorphisms
\begin{align*}
  X_{\hbar,\tau\ast\bar\tau}^\ell(\g)\cong
  \mathfrak D_{\bar\tau}^\rho(X_{\hbar,\tau}^\ell(\g)).
\end{align*}
defined by
\begin{align*}
  h_i\mapsto h_i,\quad x_i^+\mapsto \bar\tau_{ii}^{\bullet,\bullet}(0)\inv x_i^+,\quad
  x_i^-\mapsto x_i^-\quad\te{for }i\in I.
\end{align*}
\end{prop}

\begin{proof}
Since $\bar\tau_{ij}^{a,b}(z)\in\C[[z,\hbar]]$ for $a,b\in\{0,\bullet\}$ and
$\bar\tau_{ij}^{\bullet,\bullet}(z)$ is invertible,
it is straightforward to verify that
$\mathfrak D_{\bar\tau}^\rho(X_{\hbar,\tau}^\ell(\g))$
equipped with fields
\begin{align*}
  \mathfrak D_{\bar\tau}^\rho(Y)(h_i,z),\quad
  \bar\tau_{ii}^{\bullet,\bullet}(0)\inv
  \mathfrak D_{\bar\tau}^\rho(Y)(x_i^+,z),\quad
  \mathfrak D_{\bar\tau}^\rho(Y)(x_i^-,z)
\end{align*}
becomes an object in $\mathcal M_{\hbar,\tau\ast\bar\tau}^\ell(\g)$.
By Proposition \ref{prop:universal}, we get an $\hbar$-adic nonlocal vertex algebra homomorphism
\begin{align*}
  f_{\tau,\bar\tau}:V(\mathcal M_{\hbar,\tau\ast\bar\tau}^\ell(\g))\to \mathfrak D_{\bar\tau}^\rho(X_{\hbar,\tau}^\ell(\g))
\end{align*}
determined by
\begin{align*}
  f_{\tau,\bar\tau}(h_i)=h_i,\quad
  f_{\tau,\bar\tau}(x_i^+)=\bar\tau_{ii}^{\bullet,\bullet}(0)\inv x_i^+,\quad
  f_{\tau,\bar\tau}(x_i^-)=x_i^-.
\end{align*}
Recall the definition of $\wt h_i$ from \eqref{eq:def-wt-h}.
From Proposition \ref{prop:deform}, we have that
\begin{align*}
  \mathfrak D_{\bar\tau}^\rho(Y)(\wt h_i,z)
  =Y(\wt h_i,z)+\bar\tau(\wt h_i,z).
\end{align*}
Note that
\begin{align*}
  &\bar\tau(\wt h_i,z)\wt h_i
  \\
  =&q^{-r\ell\pd z}F\left(\pd z\right)
  q^{-r\ell\partial+r\ell\pd z}F\left(\partial-\pd z\right)
  \bar\tau(h_i,z)h_i\\
  =&q^{-r\ell\pd z}F\left(\pd z\right)
  q^{-r\ell\partial+r\ell\pd z}F\left(\partial-\pd z\right)\vac\ot \bar\tau_{ii}^{0,0}(z)\\
  =&\vac\ot F\left(\pd z\right)^2\bar\tau_{ii}^{0,0}(z)
  =\vac\ot F\left(\pd z\right)(\bar\tau_{ii}^{0,\bullet}(z-r\ell\hbar)
  -\bar\tau_{ii}^{0,\bullet}(z+r\ell\hbar))\\
  =&\vac\ot \log \frac{\bar\tau_{ii}^{\bullet,\bullet}(z-2r\ell\hbar)
  \bar\tau_{ii}^{\bullet,\bullet}(z+2r\ell\hbar)}
  {\bar\tau_{ii}^{\bullet,\bullet}(z)^2}.
\end{align*}
Then
\begin{align*}
  &\Res_{z_1,z_2}z_1\inv z_2\inv
  [\bar\tau(\wt h_i,z_1),Y(\wt h_i,z_2)]\\
  =&\Res_{z_1,z_2}z_1\inv z_2\inv
  Y(\bar\tau(\wt h_i,z_1-z_2)\wt h_i,z_2)\\
  =&\Res_{z_1,z_2}z_1\inv z_2\inv
  \log \frac{\bar\tau_{ii}^{\bullet,\bullet}(z_1-z_2-2r\ell\hbar)
  \bar\tau_{ii}^{\bullet,\bullet}(z_1-z_2+2r\ell\hbar)}
  {\bar\tau_{ii}^{\bullet,\bullet}(z_1-z_2)^2}\\
  =&\log \frac{\bar\tau_{ii}^{\bullet,\bullet}(-2r\ell\hbar)
  \bar\tau_{ii}^{\bullet,\bullet}(2r\ell\hbar)}
  {\bar\tau_{ii}^{\bullet,\bullet}(0)^2}.
\end{align*}
By using Baker–Campbell–Hausdorff formula, we have that
\begin{align*}
  &\exp\left(\Res_zz\inv \mathfrak D_{\bar\tau}^\rho(Y)(\wt h_i,z)\right)\\
  =&\exp\left(\Res_zz\inv Y(\wt h_i,z)\right)
  \exp\left(\Res_zz\inv \bar\tau(\wt h_i,z)\right)
  \frac{\bar\tau_{ii}^{\bullet,\bullet}(-2r\ell\hbar)^\half
  \bar\tau_{ii}^{\bullet,\bullet}(2r\ell\hbar)^\half}
  {\bar\tau_{ii}^{\bullet,\bullet}(0)}.
\end{align*}
Then we have
\begin{align*}
  &f_{\tau,\bar\tau}(E_{\tau\ast\bar\tau}(h_i))\\
  =&
  \frac{\tau_{ii}^{\bullet,\bullet}(-2r\ell\hbar)^\half}
  {\tau_{ii}^{\bullet,\bullet}(2r\ell\hbar)^\half}
  \frac{\bar\tau_{ii}^{\bullet,\bullet}(-2r\ell\hbar)^\half}
  {\bar\tau_{ii}^{\bullet,\bullet}(2r\ell\hbar)^\half}
  \exp\left(\Res_zz\inv \mathfrak D_{\bar\tau}^\rho(Y)(\wt h_i,z)\right)\vac\\
  =&
  \frac{\bar\tau_{ii}^{\bullet,\bullet}(-2r\ell\hbar)}
  {\bar\tau_{ii}^{\bullet,\bullet}(0)}
  \frac{\tau_{ii}^{\bullet,\bullet}(-2r\ell\hbar)^\half}
  {\tau_{ii}^{\bullet,\bullet}(2r\ell\hbar)^\half}
  \exp\left(\Res_zz\inv Y(\wt h_i,z)\right)\vac\\
  =&\frac{\bar\tau_{ii}^{\bullet,\bullet}(-2r\ell\hbar)}
  {\bar\tau_{ii}^{\bullet,\bullet}(0)}
  E_\tau(h_i).
\end{align*}
It follows that
\begin{align*}
  &f_{\tau,\bar\tau}(\Sing_zY(x_i^+,z)x_i^-)\\
  =&\frac{1}{2r_i\hbar}\vac z\inv
  -\frac{1}{2r_i\hbar} f_{\tau,\bar\tau}(E_{\tau\ast\bar\tau}(h_i))(z+2r\ell\hbar)\inv\\
  =&\frac{1}{2r_i\hbar}\vac z\inv
  -\frac{1}{2r_i\hbar}
  \frac{\bar\tau_{ii}^{\bullet,\bullet}(-2r\ell\hbar)}
  {\bar\tau_{ii}^{\bullet,\bullet}(0)}
  E_\tau(h_i)(z+2r\ell\hbar)\inv,
\end{align*}
and
\begin{align*}
  &\Sing_z \mathfrak D_{\bar\tau}^\rho(Y)(f_{\tau,\bar\tau}(x_i^+),z)
  f_{\tau,\bar\tau}(x_i^-)\\
  =&\bar\tau_{ii}(0)\inv\Sing_z Y(x_i^+,z)\bar\tau(\wh x_i^+,z)x_i^-\\
  =&\frac{1}{2r_i\hbar\bar\tau_{ii}(0)}\Sing_z
  \left(
    \vac \bar\tau_{ii}^{\bullet,\bullet}(z) z\inv
    -E_\tau(h_i) \bar\tau_{ii}^{\bullet,\bullet}(z)(z+2r\ell\hbar)\inv
  \right)\\
  =&\frac{1}{2r_i\hbar}\vac z\inv
  -\frac{1}{2r_i\hbar}\frac{\bar\tau_{ii}^{\bullet,\bullet}(-2r\ell\hbar)}
  {\bar\tau_{ii}^{\bullet,\bullet}(0)}
  E_\tau(h_i)(z+2r\ell\hbar)\inv.
\end{align*}
These two equations yield that the images of the elements \eqref{eq:x+0x-} and \eqref{eq:x+1x-} under $f_{\tau,\bar\tau}$ are zero.
By using the fact that $\tau_{ij}^{\bullet,\bullet}(z)\in\C[[z]]^\times$, one can straightforwardly verify that the images of the elements \eqref{eq:serre} (and \eqref{eq:def-integrable} if $X=L$) under
$f_{\tau,\bar\tau}$ are also zero.
Therefore, $f_{\tau,\bar\tau}$
factor through $X_{\hbar,\tau\ast\bar\tau}^\ell(\g)$,
which is still denoted by $f_{\tau,\bar\tau}$.
A similar argument to the proof of \cite[Proposition 4.15]{K-qaffva-rtu} shows that $f_{\tau,\bar\tau}$ is an isomorphism, as desired.
\end{proof}

%

\section{$\Z$-grading structure}\label{sec:Z-grading}

The following definition generalizes the notion of $\Z$-graded vertex algebras.

\begin{de}
A quantum vertex algebra $(V,Y,\vac,S(z))$ is called \emph{$\Z$-graded} if
\begin{align*}
  V=\bigoplus_{n\in\Z}V[n]
\end{align*}
such that
\begin{align*}
  &\vac\in V[0],\quad V[a]_nV[b]\subset V[a+b-n-1],\\
  &S_n(V[a]\ot V[b])\subset \sum_{k\in\Z}V[a+k]\ot V[b-k-n-1]\quad\te{for }a,b,n\in\Z.
\end{align*}
\end{de}

\begin{rem}
Let $V$ be a vertex operator algebra with conformal vector $\omega$, and set
\begin{align*}
  \sum_{n\in\Z}L(n)z^{-n-2}=Y(\omega,z).
\end{align*}
View $V$ as a quantum vertex algebra with trivial quantum Yang-Baxter operator.
Then $V$ is $\Z$-graded with 
\begin{align*}
  V[n]=\set{v\in V}{L(0)v=nv}.
\end{align*}
\end{rem}

In this section, we show that both $V_{\hbar,\varepsilon}^\ell(\g)$ and $L_{\hbar,\varepsilon}^\ell(\g)$ contain dense $\Z$-graded quantum vertex subalgebras. 
This is achieved by proving the existence of linear transformations of the following type:

\begin{de}
A $\C$-linear map $L(0)$ on an $\hbar$-adic quantum vertex algebra $(V,Y,\vac,S(z))$ is called a
\emph{degree operator} if $L(0)\vac=0$,
\begin{align}
  &[L(0),Y(u,z)]=Y(L(0) u,z)+z\pd zY(u,z)\quad\te{for }u\in V,\label{eq:L0-Y-comp}\\
  &[L(0)\ot 1+1\ot L(0),S(z)]=z\pd zS(z),
\end{align}
\end{de}

\begin{rem}
Let $(V,Y,\vac,S(z))$ be an $\hbar$-adic quantum vertex algebra containing a dense $\Z$-graded quantum vertex subalgebra $V'$.
Define $L_\hbar$ on $V'$ be letting $L(0)=n$ on $V'[n]$,
and extend $L(0)$ to the whole space $V$.
Then $L(0)$ is a degree operator of $V$.
\end{rem}

\begin{rem}
Let $U$ and $V$ be $\hbar$-adic quantum vertex algebras with degree operators $L_U(0)$ and $L_V(0)$, respectively. Then
\begin{align*}
  L_{U\wh\ot V}(0)=L_U(0)\ot 1+1\ot L_V(0)
\end{align*}
is a degree operator of
the tensor product $\hbar$-adic quantum vertex algebra $U\wh \ot V$.
\end{rem}


Recall the deforming triple $(H,\rho,\wh\varepsilon)$ of $V(\mathcal M^\ell(\g))$ from Section \ref{sec:qaffva}.
The following result establish the existence of a degree operator on both $V(\mathcal M^\ell(\g))$ and $H$ that are compatible with $\rho$ and $\wh\varepsilon$.

\begin{lem}\label{lem:L0-mod-com}
There exists degree operators $L_\hbar(0)$ on both $V(\mathcal M^\ell(\g))$ and $H$ such that
\begin{align*}
  &L_\hbar(0)\hbar=-\hbar,\quad L_\hbar(0)h_i=h_i,\quad L_\hbar(0)x_i^\pm=x_i^\pm,\\
  &L_\hbar(0)\wh h_i=\wh h_i,\quad L_\hbar(0)\wh x_i^\pm=0\quad\te{for }i\in I.
\end{align*}
Moreover,
\begin{align*}
  &\Delta(L_\hbar(0)h)=(L_\hbar(0)\ot 1+1\ot L_\hbar(0))\Delta(h)\quad\te{for }h\in H,\\
  &\rho(L_\hbar(0)u)=(L_\hbar(0)\ot 1+1\ot L_\hbar(0))\rho(u)\quad\te{for }u\in V(\mathcal M^\ell(\g)).
\end{align*}
Furthermore, for $h\in H$ and $v\in V(\mathcal M^\ell(\g))$, we have that
\begin{align}
  &[L_\hbar(0),\wh\varepsilon(h,z)]v
  =\wh\varepsilon(L_\hbar(0)h,z)v
  +z\pd z\wh\varepsilon(h,z)v.
  \label{eq:L0-mod-com}
\end{align}
\end{lem}

\begin{proof}
Remark \ref{rem:module-realization} implies that there exists a degree operator $L(0)$ on
$V(\mathcal M^\ell(\g))$ such that
\begin{align*}
  L(0)\hbar=0,\quad L(0)h_i=h_i,\quad L(0)x_i^\pm=x_i^\pm\quad
  \te{for }i\in I.
\end{align*}
Similarly, the definition of $H$ yields a degree operator on $H$, still denoted by $L(0)$,
such that
\begin{align*}
  L(0)\hbar=0,\quad L(0)\wh h_i=\wh h_i,\quad
  L(0)\wh x_i^\pm=0\quad\te{for }i\in I.
\end{align*}
Note that $\hbar\pd\hbar$ is a derivation on both $H$ and $V(\mathcal M^\ell(\g))$.
Moreover,
\begin{align*}
  \Delta\big((\hbar\pd\hbar)h\big) &= \big(\hbar\pd\hbar\otimes 1+1\otimes\hbar\pd\hbar\big)\Delta(h) \quad\te{for }h\in H,\\
  \rho\big((\hbar\pd\hbar)u\big) &= \big(\hbar\pd\hbar\otimes 1+1\otimes\hbar\pd\hbar\big)\rho(u) \quad\te{for }u\in V(\mathcal M^\ell(\g)).
\end{align*}
It is straightforward to verify that $L_\hbar(0)=L(0)-\hbar\pd\hbar$ is a degree operator on both $V(\mathcal M^\ell(\g))$ and $H$, and that it satisfies the moreover statement.

For the further statement, regard $\hbar$ as an element of both $H$ and $V(\mathcal M^\ell(\g))$.
One then straightforwardly verifies that the relation \eqref{eq:L0-mod-com} holds for
\begin{align*}
  h\in T=\set{\wh h_i,\,\wh x_i^\pm,\,\hbar}{i\in I},\quad
  v\in S=\set{h_i,\,x_i^\pm,\,\hbar}{i\in I}.
\end{align*}
Since $H$ is generated by $T$ and $V(\mathcal M^\ell(\g))$ is generated by $S$,
the relation \eqref{eq:L0-mod-com} extends to all of $H$ and $V(\mathcal M^\ell(\g))$, which proves the desired further statement.
\end{proof}

Note that $\wh\varepsilon$ is invertible, and
\begin{align*}
  &\wh\varepsilon\inv(\wh h_i,z)h_j=-\vac\ot
  \left([r_ia_{ij}]_{q^{\pd z}}[r\ell]_{q^{\pd z}}q^{r\ell \pd z}z^{-2}-r_ia_{ij}r\ell z^{-2}\right),\\
  &\wh\varepsilon\inv(\wh h_i,z)x_j^\pm
  =\mp x_j^\pm\ot \left([r_ia_{ij}]_{q^{\pd z}}q^{r\ell\pd z}z\inv-r_ia_{ij}z\inv\right),\\
  &\wh\varepsilon\inv(\wh x_i^\pm,z)h_j=h_j\ot 1
  \pm \vac\ot \left( [r_ia_{ij}]_{q^{\pd z}}q^{r\ell\pd z}z\inv-r_ia_{ij}z\inv \right),\\
  &\wh\varepsilon\inv(\wh x_i^\pm,z)x_j^\pm
  =x_j^\pm\ot z(z-r_ia_{ij}\hbar)\inv,\\
  &\wh\varepsilon\inv(\wh x_i^+,z)x_j^-\ot z^{\delta_{ij}}(z+2r\ell\hbar)^{-\delta_{ij}},\\
  &\wh\varepsilon\inv(\wh x_i^-,z)x_j^+
  =x_j^+\ot z^{\delta_{ij}}(z-2r\ell\hbar)^{-\delta_{ij}}
  \frac{z+r_ja_{ji}\hbar}{z-r_ja_{ji}\hbar}.
\end{align*}
Similar to Lemma \ref{lem:L0-mod-com}, we have the following result.
\begin{lem}\label{lem:L0-mod-com-inv}
For $h\in H$ and $v\in V(\mathcal M^\ell(\g))$, we have that
\begin{align}\label{eq:L0-mod-com-inv}
  [L_\hbar(0),\wh\varepsilon\inv(h,z)]v=\wh\varepsilon\inv(L_\hbar(0)h,z)v
  +z\pd z\wh\varepsilon\inv(h,z)v.
\end{align}
\end{lem}

The following result shows that there exists a degree operator on both $V_{\hbar,\varepsilon}^\ell(\g)$ and $L_{\hbar,\varepsilon}^\ell(\g)$.

\begin{thm}\label{thm:Lh0}
Let $X=V$ or $L$.
There exists a degree operator $L_\hbar(0)$ 
on $X_{\hbar,\varepsilon}^\ell(\g)$ 
such that
\begin{align}
  &L_\hbar(0)\hbar=-\hbar,\quad
  L_\hbar(0)h_i=h_i,\quad
  L_\hbar(0)x_i^\pm=x_i^\pm
  \quad\te{for }i\in I.\label{eq:Lh0}
\end{align}
\end{thm}

\begin{proof}
View $V(\mathcal M_{\hbar,\varepsilon}^\ell(\g))$
as $\mathfrak D_{\wh\varepsilon}^\rho(V(\mathcal M^\ell(\g)))$ by Lemma \ref{lem:deform-realization}.
By using Lemma \ref{lem:L0-mod-com}, one can straightforwardly verify that the degree operator $L_\hbar(0)$ on $V(\mathcal M^\ell(\g))$ satisfies \eqref{eq:Lh0} and the following relation
\begin{align*}
  [L_\hbar(0),\mathfrak D_{\wh\varepsilon}^\rho(Y)(u,z)]
  =\mathfrak D_{\wh\varepsilon}^\rho(Y)(L_\hbar(0)u,z)
  z\pd z \mathfrak D_{\wh\varepsilon}^\rho(Y)(u,z)
  \quad\te{for }u\in \mathfrak D_{\wh\varepsilon}^\rho(V(\mathcal M^\ell(\g))).
\end{align*}
Combining Corollary \ref{coro:deform-qva} and Lemmas
\ref{lem:L0-mod-com}, \ref{lem:L0-mod-com-inv}, we immediate get that
\begin{align*}
  [L_\hbar(0)\ot 1+1\ot L_\hbar(0),S_\varepsilon(z)]
  =z\pd z S_\varepsilon(z).
\end{align*}
Note that the elements \eqref{eq:x+0x-}, \eqref{eq:x+1x-} and \eqref{eq:serre} are preserved by $L_\hbar(0)$.
Then $L_\hbar(0)$ induces a degree operator on $X_{\hbar,\varepsilon}^\ell(\g)$, as desired.
\end{proof}

Since both $V_{\hbar,\varepsilon}^\ell(\g)$ and $L_{\hbar,\varepsilon}^\ell(\g)$ are generated by $\set{h_i,x_i^\pm}{i\in I}$, we get the following immediate consequence of Theorem \ref{thm:Lh0}.
\begin{coro}\label{coro:Lh0}
Let $X=V$ or $L$.
Then $X_{\hbar,\varepsilon}^\ell(\g)$ contains a $\Z$-graded quantum vertex subalgebra \begin{align*}
  X_{\varepsilon}^\ell(\g):= \bigoplus_{n\in\Z}\set{u\in X_{\hbar,\varepsilon}^\ell(\g)}{L_\hbar(0)u=nu}.
\end{align*}
\end{coro}

\section{Filtration structure}\label{sec:filtration}

The following notion is a generalization of that of vertex algebras.

\begin{de}
A \emph{increasing filtration} of a quantum vertex algebra $(V,Y,\vac,S(z))$ is a list of subspaces $\{F^n\}_{n\in\Z}$, such that
$\vac\in F^0$ and
\begin{align*}
  &\cdots \subset F^n\subset F^{n+1}\subset \cdots\quad\te{for }n\in\Z,\\
  &\bigcap_{n\in\Z}F^n=0,\quad \bigcup_{n\in\Z}F^n=V,\quad
  F^a_nF^b\subset F^{a+b-n-1},\\
  &S_n(F^a\ot F^b)\subset \sum_{k\in\Z}
  F^{a+k}\ot F^{b-k-n-1}\quad\te{for }a,b,n\in\Z.
\end{align*}
\end{de}

The following result is also a generalization of that of vertex algebras.

\begin{prop}
Let $(V,S(z))$ be a quantum vertex algebra with a increasing filtration $\{F^n\}_{n\in\Z}$. Define
\begin{align*}
  \fgr V=\bigoplus_{n\in\Z}F^n/F^{n-1}.
\end{align*}
Then $\fgr V$ carries a $\Z$-graded quantum vertex algebra structure with vacuum $\vac+F^{-1}$, the
vertex operator map
\begin{align*}
  Y^{\fgr}(u,z)v=\sum_{n\in\Z}u_n^{\fgr}vz^{-n-1}
\end{align*}
defined by
\begin{align*}
  (u+F^{s-1})_n^{\fgr}(v+F^{t-1})=u_nv+F^{s+t-n-2}
  \quad\te{for }u\in F^s,\,v\in F^t,\,n\in\Z,
\end{align*}
and the quantum Yang-Baxter operator defined by
\begin{align*}
  (v+F^{t-1})\ot (u+F^{s-1})\mapsto \sum_{n\in\Z}
  \left(S_n(v\ot u)+\sum_{k\in\Z}F^{t-1+k}\ot F^{s-k-n-1}
  \right)z^{-n-1}.
\end{align*}
\end{prop}

Let $X=V$ or $L$ and let $\tau\in \mathfrak T$.
In this section, we prove that $X_{\hbar,\tau}^\ell(\g)$
contains a dense quantum vertex subalgebra which admits a increasing filtration. We also show that the associated graded algebra of this dense quantum vertex subalgebra is isomorphic to $X_\varepsilon^\ell(\g)$.
%

Recall from Lemma \ref{lem:V-h-tau-deform-triple} that $\bar H$ is defined as the
$\hbar$-adic vertex quotient bialgebra of
$H$ modulo the ideal preserved by
$L_\hbar(0)$. The degree operator $L_\hbar(0)$ on
$H$ therefore descends to a degree operator on $\bar H$, which we denote by the same symbol.
Since 
\begin{align*}
  \bigoplus_{n\in\Z}\set{u\in \bar H}{L_\hbar(0)u=mu}
\end{align*}
is dense in $\bar H$, we immediate get the following result.

\begin{lem}
For $n\in\Z$, denote by $\bar H^n$ the closure of the following space under $\hbar$-adic topology
\begin{align*}
  \bigoplus_{m\le n}\set{u\in \bar H}{L_\hbar(0)u=mu}.
\end{align*}
Then $\mathop{\bigcup}\limits_{n\in\Z}\bar H^n$
is a dense vertex subalgebra of $\bar H$, and
$\{\bar H^n\}_{n\in\Z}$ is a increasing filtration of $\mathop{\bigcup}\limits_{n\in\Z}\bar H^n$.
Moreover,
\begin{align*}
  &\Delta(\bar H^n)\subset \sum_{k\in\Z}\bar H^k\ot \bar H^{n-k}
  \quad \te{and}\quad \varepsilon(\bar H^n)\subset \delta_{n,0}\C\quad\te{for }n\in\Z.
\end{align*}
\end{lem}

From Corollary \ref{coro:Lh0}, we immediate get the following result.

\begin{lem}
For $n\in\Z$ and $X=V$ or $L$, denote by $X^n$ the closure of the following space under $\hbar$-adic topology
\begin{align*}
  \bigoplus_{m\le n}\set{u\in X_{\hbar,\varepsilon}^\ell(\g)}{L_\hbar(0)u=mu}.
\end{align*}
Then $\mathop{\bigcup}\limits_{n\in\Z}X^n$
is a dense quantum vertex subalgebra of $X_{\hbar,\varepsilon}^\ell(\g)$ and
$\{X^n\}_{n\in\Z}$ defines a increasing filtration of $\mathop{\bigcup}\limits_{n\in\Z}X^n$.
\end{lem}

Let $\tau\in\mathfrak T$ and let $X=V$ or $L$.
Recall the
deforming triple $(\bar H,\rho,\tau)$ of $X_{\hbar,\varepsilon}^\ell(\g)$ from Lemma \ref{lem:V-h-tau-deform-triple}.
The following result shows that the increasing filtrations are compatible with $\rho$ and $\tau$.

\begin{lem}\label{lem:rho-Lh0-compatible}
Let $X=V$ or $L$.
For $u\in X_{\hbar,\varepsilon}^\ell(\g)$, we have that
\begin{align}\label{eq:rho-Lh0-compatible}
  \rho(L_\hbar(0)u)=(L_\hbar(0)\ot 1+1\ot L_\hbar(0))\rho(u).
\end{align}
For $a,b\in\Z$, $h\in \bar H^a$, $u\in X^b$, and $n\in\Z$, we have that
\begin{align}\label{eq:tau-filter-com}
  \Res_zz^n\tau(h,z)u\in \delta_{n,-1}\varepsilon(h)u
  +X^{a+b-n-2}
  \subset X^{a+b-n-1}.
\end{align}
\end{lem}

\begin{proof}
It is easy to see that \eqref{eq:rho-Lh0-compatible} holds for
\begin{align*}
  u\in S=\set{h_i,\,x_i^\pm}{i\in I}\uplus\{\vac\},
\end{align*}
and \eqref{eq:tau-filter-com} holds for $u\in S$, $n\in\Z$ and
\begin{align*}
  h\in T=\set{\wh h_i,\,\wh x_i^\pm}{i\in I}\uplus\{1\}.
\end{align*}
The lemma now follows, since
$X_{\hbar,\varepsilon}^\ell(\g)$ is generated by $S$ and $\bar H$ is generated by $T$.
\end{proof}

Lemma \ref{lem:rho-Lh0-compatible} immediately yields the following result.

\begin{thm}
Let $\tau\in\mathfrak T$, and let $X=V$ or $L$.
Then $\mathop{\bigcup}\limits_{n\in\Z}X^n$
is a quantum vertex subalgebra of $X_{\hbar,\tau}^\ell(\g)$ and
$\{X^n\}_{n\in\Z}$ defines a increasing filtration of $\mathop{\bigcup}\limits_{n\in\Z}X^n$.
\end{thm}

%
%
%
%

The associated graded algebras of $\mathop{\bigcup}\limits_{n\in\Z}X^n$ are described in the following result.

\begin{thm}
Let $X=V$ or $L$.
There is a quantum vertex algebra isomorphism from
$\fgr \Big(\mathop{\bigcup}\limits_{n\in\Z}X^n\Big)$ to $X_\varepsilon^\ell(\g)$
defined by
\begin{align*}
  \hbar+X^{-2}\mapsto \hbar,\quad
  h_i+X^0\mapsto h_i,\quad x_i^\pm+X^0\mapsto x_i^\pm
  \quad\te{for }i\in I.
\end{align*}
\end{thm}

\begin{proof}
Since $L_\hbar(0)\hbar=-\hbar$, we get the following relation
\begin{align*}
  X^n=X^{n-1}\bigoplus\set{u\in X_{\hbar,\varepsilon}^\ell(\g)}{L_\hbar(0)u=nu}.
\end{align*}
Then the following map defines a linear isomorphism $f:\fgr \Big(\mathop{\bigcup}\limits_{n\in\Z}X^n\Big)\to X_\varepsilon^\ell(\g)$
\begin{align*}
  u+X^{n-1}\mapsto u\quad\te{for }u\in \set{u\in X_{\hbar,\varepsilon}^\ell(\g)}{L_\hbar(0)u=nu},\,n\in\Z.
\end{align*}
View $X_{\hbar,\tau}^\ell(\g)$ as $\mathfrak D_\tau^\rho(X_{\hbar,\varepsilon}^\ell(\g))$.
From Lemma \ref{lem:rho-Lh0-compatible}, we have that
\begin{align*}
  &(u+X^{s-1})_n^{\fgr}(v+X^{t-1})\\
  =&\Res_zz^n \mathfrak D_\tau^\rho(Y)(u,z)v
  +X^{s+t-n-2}\\
  =&\Res_zz^n\sum Y(u_{(1)},z)\tau(u_{(2)},z)v
  +X^{s+t-n-2}\\
  =&\Res_zz^n\sum Y(u_{(1)},z)\varepsilon(u_{(2)})v
  +X^{s+t-n-2}\\
  =&\Res_zz^n\sum Y(u,z)v+X^{s+t-n-2}\\
  =&u_nv+X^{s+t-n-2}\quad\te{for }u\in X^s,\,v\in X^t,\,s,t,n\in\Z.
\end{align*}
Therefore, $f$ is a quantum vertex algebra isomorphism, as desired.
\end{proof}

\section{Related to twisted quantum affine algebras}\label{sec:tqaff}

In this section, we choose a specific $\tau\in\mathfrak T$ and establish an isomorphism between 
the category of equivariant $\phi$-coordinated quasi $V_{\hbar,\tau}^\ell(\g)$-modules and the category of restricted modules of twisted quantum affinization algebras of level $\ell$.

Let $\mu$ be a permutation of $I$ of order $N$, such that
\begin{align*}
  a_{\mu(i),\mu(j)}=a_{ij}\quad\te{for }i,j\in I.
\end{align*}
Then
\begin{align*}
  r_{\mu(i)}=r_i\quad\te{for }i\in I.
\end{align*}
Fix a $N$-th primitive root of unity $\xi$.
Set
\begin{align*}
  f_s(z)=
  \begin{cases}
    \log\left((e^{z/2}-e^{-z/2})/z\right),& s\equiv 0\mod N,\\
    \log\left((e^{z/2}-\xi^se^{-z/2})/(1-\xi^s)\right),& s\not\equiv 0\mod N.
  \end{cases}
\end{align*}
Define
\begin{align*}
  &\tau_{ij}^{0,0}(z)=\sum_{s\in\Z_N}
  q_i^{-a_{i,\mu^s(j)}\pd z}
  [r\ell]_{q^{\pd z}}^2\pdiff{z}{2}f_s(z),\\
  &\tau_{ij}^{\bullet,0}(z)=\tau_{ij}^{0,\bullet}(z)
  =-\sum_{s\in\Z_N}
  q_i^{-a_{i,\mu^s(j)}\pd z}[r\ell]_{q^{\pd z}}\pd zf_s(z),\\
  &\tau_{ij}^{\bullet,\bullet}(z)
  =\prod_{s\in\Z_N}\exp\left(f_s(z-r_ia_{i,\mu^s(j)}\hbar)\right).
\end{align*}
Then the relations \eqref{tau1}, \eqref{tau2}, \eqref{tau3}, \eqref{tau3.5} and \eqref{tau4} specialize to the following relations.
\begin{align}
  \tag{$\tau$1'}\label{tau1'}
  &[h_i(z_1),h_j(z_2)]
  =
  \sum_{s\in\Z_N}[r_ia_{i,\mu^s(j)}]_{q^{\pd{z_2}}}
  [r\ell]_{q^{\pd{z_2}}}\\
  &\quad\times\nonumber
  (q^{-r\ell\pd{z_2}}\iota_{z_1,z_2}
    -q^{r\ell\pd{z_2}}\iota_{z_2,z_1})
  \frac{\xi^se^{-z_1+z_2}}{(1-\xi^se^{-z_1+z_2})^2},\\
  \tag{$\tau$2'}\label{tau2'}
  &[h_i(z_1),x_j^\pm(z_2)]=
  \sum_{s\in\Z_N}[r_ia_{i,\mu^s(j)}]_{q^{\pd{z_2}}}\\
  &\quad\times\nonumber
  (q^{-r\ell\pd{z_2}}\iota_{z_1,z_2}-q^{r\ell\pd{z_2}}\iota_{z_2,z_1})
  \frac{1+\xi^se^{-z_1+z_2}}{2-2\xi^se^{-z_1+z_2}},\\
  \label{tau3'}\tag{$\tau$3'}
  &x_i^+(z_1)x_j^-(z_2)-\iota_{z_2,z_1}
  \prod_{s\in\Z_N}\frac{1-\xi^sq_i^{a_{i,\mu^s(j)}}e^{z_2-z_1}}
  {q_i^{a_{i,\mu^s(j)}}-\xi^se^{z_2-z_1}}
  x_j^-(z_2)x_i^+(z_1)\\
  &\nonumber\quad =\frac{\delta_{ij}}{2r_i\hbar}z_1\inv\delta\left(\frac{z_2}{z_1}\right)
  -\frac{\delta_{ij}}{2r_i\hbar}
  E_\tau\big(h_j(z_2)\big)
  z_1\inv\delta\left(\frac{q^{-2r\ell}z_2}{z_1}\right),\\
  \tag{$\tau$3.5'}\label{tau3.5'}
  &(z_1-z_2)^{\delta_{ij}}
  (z_1-z_2+2r\ell\hbar)^{\delta_{ij}}
  \bigg(x_i^+(z_1)x_j^-(z_2)\\
  &\quad\nonumber-\iota_{z_2,z_1}
  \prod_{s\in\Z_N}\frac{1-\xi^sq_i^{a_{i,\mu^s(j)}}e^{z_2-z_1}}
  {q_i^{a_{i,\mu^s(j)}}-\xi^se^{z_2-z_1}}
  x_j^-(z_2)x_i^+(z_1)\bigg)=0,\\
  \tag{$\tau$4'}\label{tau4'}
  &\prod_{s\in\Z_N}(1-\xi^sq_i^{a_{i,\mu^s(j)}}e^{-z_1+z_2})
  x_i^\pm(z_1)x_j^\pm(z_2)\\
  &\quad\nonumber=
  \prod_{s\in\Z_N}(q_i^{a_{i,\mu^s(j)}}-\xi^se^{-z_1+z_2})
  x_j^\pm(z_2)x_i^\pm(z_1).
\end{align}

\begin{de}
Let $G$ be a group. A {\em $G$-module $\hbar$-adic nonlocal vertex algebra} is an
$\hbar$-adic nonlocal vertex algebra $V$ on which $G$ acts as an automorphism group.
\end{de}

\begin{prop}
For $X=V$ or $L$ and $s\in\Z_N$, there exists an $\hbar$-adic nonlocal vertex algebra endomorphism
$R(s)$ on $X_{\hbar,\tau}^\ell(\g)$ defined by
\begin{align*}
  R(s)h_i=h_{\mu^s(i)},\quad R(s)x_i^\pm=x_{\mu^s(i)}^\pm
  \quad\te{for }i\in I.
\end{align*}
Moreover, $\Z_N$ acts an $X_{\hbar,\tau}^\ell(\g)$ as an automorphism group, making $X_{\hbar,\tau}^\ell(\g)$ a $\Z_N$-module $\hbar$-adic nonlocal vertex algebra.
\end{prop}

\begin{proof}
Note that the relations \eqref{tau1'}, \eqref{tau2'}, \eqref{tau3.5'} and \eqref{tau4'} hold with
\begin{align}\label{eq:mu-assign}
  h_i(z)=Y(h_{\mu^s(i)},z),\quad x_i^\pm(z)=Y(x_{\mu^s(i)}^\pm,z)\quad\te{for }i\in I.
\end{align}
Since
\begin{align*}
  r_{\mu^s(i)}=r_i,\quad \tau_{\mu^s(i),\mu^s(i)}^{\bullet,\bullet}(z)=\tau_{ii}^{\bullet,\bullet}(z)
  \quad\te{for }i\in I,
\end{align*}
we get that the relations \eqref{eq:x+0x-}, \eqref{eq:x+1x-}, \eqref{eq:serre} and \eqref{eq:def-integrable} hold under the assignment \eqref{eq:mu-assign}.
By applying Proposition \ref{prop:universal-qaff}, we prove the first statement.
Note that $R(s)R(t)u=R(s+t)u$ and $R(s)R(-s)u=u$ for any $s,t\in\Z_N$ and
$u$ in the generating subset $\set{h_i,x_i^\pm}{i\in I}$
of $X_{\hbar,\tau}^\ell(\g)$.
Then $R(s)$ is invertible and $R$ defines a $\Z_N$-action on $X_{\hbar,\tau}^\ell(\g)$, as desired.
\end{proof}

Set
\begin{align}
\phi(x,z)=e^{z x\frac{d}{dx}}x=xe^z,
\end{align}
which is a particular associate (see \cite{Li-phi-coor}) of the one-dimensional additive formal group law
$F_a(x,y):=x+y\in \C[[x,y]]$.

\begin{de}
Let $V$ be a $G$-module $\hbar$-adic nonlocal vertex algebra and
let $\chi_{\phi}$ be a linear character of $G$.
A {\em $(G,\chi_{\phi})$-equivariant $\phi$-coordinated quasi $V$-module} is a
topologically free $\C[[\hbar]]$-module $W$
equipped with a $\C[[\hbar]]$-module map
\begin{align*}
  Y_W^\phi(\cdot,z):\  V\longrightarrow \E_\hbar(W);\ \ \
  v\mapsto Y_W^{\phi}(v,z),
\end{align*}
satisfying the conditions that $Y_W^\phi(\vac,z)=1_W$,
\begin{align}\label{eq:G-equiv}
  Y_W^\phi(R(g)u,z)=Y_W^\phi(u,\chi_\phi(g)\inv z)\quad\te{for }g\in G,\,u\in V
\end{align}
and that for $u,v\in V$, $\left(Y_W^\phi(u,z),Y_W^\phi(v,z)\right)$ is $\chi_\phi(G)$-compatible and
\begin{align*}
  Y_\E^\phi\left(Y_W^\phi(u,z),z_0\right)Y_W^\phi(v,z)
  =Y_W^\phi\left(Y(u,z_0)v,z_0\right).
\end{align*}
\end{de}

Let $\Gamma$ be a subgroup of $\C^\times$.
Define a group homomorphism
\begin{align*}
R:\Gamma\rightarrow \te{GL}(\E_\hbar(W)),\  \  g\mapsto R_g;\quad
R_g(a(z))=a(g^{-1} z)\quad  \te{ for }a(z)\in \E_\hbar(W).
\end{align*}
Let $\imath$ be the natural embedding of $\Gamma$ into $\C^{\times}$.
The following result was given in \cite[Theorem 2.21]{JKLT-G-phi-mod}.

\begin{thm}\label{thm:abs-construct}
Let $W$ be a topologically free $\C[[\hbar]]$-module,
and let $\Gamma$ be a subgroup of $\C^\times$.
Suppose that $U$ is a $\Gamma$-compatible $\Gamma$-subset of $\E_\hbar(W)$.
Then
there exists a unique minimal $Y_\E^\phi$-closed quasi-compatible $\C[[\hbar]]$-submodule $\<U\>_\phi$
such that
\begin{align*}
  &\{1_W\}\cup U\subset \<U\>_\phi,\quad [\<U\>_\phi]=\<U\>_\phi,\quad\te{and}\quad \overline{\<U\>_\phi}=\<U\>_\phi.
\end{align*}
Moreover, $(\<U\>_\phi,Y_\E^\phi,1_W,R)$ carries the structure of a $\Gamma$-module $\hbar$-adic nonlocal vertex
algebra and $W$ is a $(\Gamma,\imath)$-equivariant $\phi$-coordinated quasi $\<U\>_\phi$-module.
\end{thm}

\begin{de}
Let $\mathcal M_\ell^\phi$ be the category consisting of topologically free $\C[[\hbar]]$-modules $W$, equipped with fields $h_{i,q}(z),x_{i,q}^\pm(z)\in\E_\hbar(W)$ ($i\in I$) satisfies the following relations
\begin{align}
  &\tag{$\phi$0}\label{phi0}h_{\mu(i),q}(z)=h_{i,q}(\xi\inv z),\quad x_{\mu(i),q}^\pm(z)=x_{i,q}^\pm(\xi\inv z),\\
  &\tag{$\phi$1}\label{phi1}
  [h_{i,q}(z_1),h_{j,q}(z_2)]=
  \sum_{s\in\Z_N}[r_ia_{i,\mu^s(j)}]_{q^{z_2\pd{z_2}}}
  [r\ell]_{q^{z_2\pd{z_2}}}\\
  &\quad\nonumber\times
  (q^{-r\ell z_2\pd{z_2}}\iota_{z_1,z_2}
  -q^{r\ell z_2\pd{z_2}}\iota_{z_2,z_1})
  \frac{\xi^sz_2/z_1}{(1-\xi^sz_2/z_1)^2},\\
  &\tag{$\phi$2}\label{phi2}
  [h_{i,q}(z_1),x_{j,q}^\pm(z_2)]=\pm x_{j,q}^\pm(z_2)
  \sum_{s\in\Z_N}[r_ia_{i,\mu^s(j)}]_{q^{z_2\pd{z_2}}}\\
  &\quad\nonumber\times
  (q^{-r\ell z_2\pd{z_2}}\iota_{z_1,z_2}
  -q^{r\ell z_2\pd{z_2}}\iota_{z_2,z_1})
  \frac{1+\xi^sz_2/z_1}{2-2\xi^sz_2/z_1},\\
  &\tag{$\phi$3.5}\label{phi3.5}
  \prod_{s\in\Z_N}(1-\xi^sz_2/z_1)^{\delta_{i,\mu^s(j)}}
  (1-q^{-2r\ell}\xi^sz_2/z_1)^{\delta_{i,\mu^s(j)}}
  \bigg(x_{i,q}^+(z_1)x_{j,q}^-(z_2)
  \\
  &\quad\nonumber
  -
  \iota_{z_2,z_1}\prod_{s\in\Z_N}\frac{1-\xi^sq_i^{a_{i,\mu^s(j)}}z_2/z_1}
    {q_i^{a_{i,\mu^s(j)}}-\xi^sz_2/z_1}
  x_{j,q}^-(z_2)x_{i,q}^+(z_1)\bigg)=0,\\
  &\tag{$\phi$4}\label{phi4}
  \prod_{s\in\Z_N}(1-\xi^sq_i^{a_{i,\mu^s(j)}}z_2/z_1)
  x_{i,q}^\pm(z_1)x_{j,q}^\pm(z_2)\\
  &\quad\nonumber
  =\prod_{s\in\Z_N}(q_i^{a_{i,\mu^s(j)}}-\xi^sz_2/z_1)
  x_{j,q}^\pm(z_2)x_{i,q}^\pm(z_1).
\end{align}
\end{de}

\begin{lem}\label{lem:U-W}
Let $(W,h_{i,q}(z),x_{i,q}^\pm(z))\in\obj\mathcal M_\ell^\phi$,
and let $\chi_\phi$ be the linear character of $\Z_N$ defined by $s\mapsto \xi^s$.
Then there exists a $\Z_N$-module $\hbar$-adic nonlocal vertex algebra $V_W\subset\E_\hbar(W)$ generated by
\begin{align}\label{eq:U-W}
  U_W=\set{h_{i,q}(z),x_{i,q}^\pm(z)}{i\in I,s\in\Z_N},
\end{align}
with $\Z_N$-action defined by
\begin{align*}
  R:\Z_N\to \Aut(V_W), \quad
  R(s)a(z)=a(\xi^{-s}z)\quad\te{for }a(z)\in V_W,\,s\in\Z_N.
\end{align*}
Moreover, $W$ is a $(\Z_N,\chi_\phi)$-equivariant $\phi$-coordinated quasi $V_W$-module with module action $Y_W^\phi$ determined by
\begin{align}\label{eq:U-W-mod}
  Y_W^\phi(h_{i,q}(z),z_0)=h_{i,q}(z_0),\quad
  Y_W^\phi(x_{i,q}^\pm(z),z_0)=x_{i,q}^\pm(z_0)\quad\te{for }i\in I.
\end{align}
\end{lem}

\begin{proof}
It follows immediately from \eqref{phi0}, \eqref{phi1}, \eqref{phi2}, \eqref{phi3.5} and \eqref{phi4} that $U_W$ is a $\chi_\phi(\Z_N)$-compatible $\chi_\phi(\Z_N)$-subset of $\E_\hbar(W)$.
Applying Theorem \ref{thm:abs-construct} to $U_W$, we obtain a $\chi_\phi(\Z_N)$-module $\hbar$-adic nonlocal vertex algebra $\< U_W\>_\phi$, together with a $(\chi_\phi(\Z_N),\imath)$-equivariant $\phi$-coordinated quasi $\< U_W\>_\phi$-module structure on $W$ defined by \eqref{eq:U-W-mod}.
Set $V_W = \langle U_W\rangle_\phi$.
Since $\chi_\phi$ is a group homomorphism, $V_W$ is the required $\Z_N$-module $\hbar$-adic nonlocal vertex algebra, and $W$ is the required $(\Z_N,\chi_\phi)$-equivariant $\phi$-coordinated quasi $V_W$-module.
\end{proof}

%

\begin{de}
Define $\mathcal R_\ell^\phi$ to be the full subcategory of $\mathcal M_\ell^\phi$ consisting of objects $(W,h_{i,q}(z),x_{i,q}^\pm(z))$ such that
\begin{align}
  &\tag{$\phi$3}\label{phi3}
  x_{i,q}^+(z_1)x_{j,q}^-(z_2)-
  \iota_{z_2,z_1}\prod_{s\in\Z_N}\frac{1-\xi^sq_i^{a_{i,\mu^s(j)}}z_2/z_1}
    {q_i^{a_{i,\mu^s(j)}}-\xi^sz_2/z_1}
    x_{j,q}^-(z_2)x_{i,q}^+(z_1)\\
  &\quad\nonumber=\sum_{s\in\Z_N}
  \frac{\delta_{i,\mu^s(j)}}{2r_i\hbar}
  \left(\delta\left(\frac{\xi^sz_2}{z_1}\right)
    -E_\tau(h_{j,q}(z_2))\delta\left(\frac{q^{-2r\ell}\xi^sz_2}{z_1}\right)\right),\\
  &\left(\left(x_{i,q}^\pm(z)\right)_0^\phi\right)^{m_{ij}}
  x_{j,q}^\pm(z)=0,\quad\te{if }a_{ij}\le 0.\tag{$\phi$5}\label{phi5}
\end{align}
\end{de}

\begin{prop}\label{prop:equiv-mods=R}
Let $\chi_\phi$ be a linear character of $\Z_N$ defined by $s\mapsto \xi^s$.
The category of $(\Z_N,\chi_\phi)$-equivariant $\phi$-coordinated quasi $V_{\hbar,\tau}^\ell(\g)$-modules is isomorphic $\mathcal R_\ell^\phi$.
\end{prop}

\begin{proof}
Let $(W,h_i(z),x_i^\pm(z))$ be an object of $\mathcal R_\ell^\phi$.
By Lemma \ref{lem:U-W}, we get a $\Z_N$-module $\hbar$-adic nonlocal vertex algebra $V_W$, together with a $(\Z_N,\chi_\phi)$-equivariant $\phi$-coordinated quasi $V_W$-module structure $Y_W^\phi$ on $W$ defined by
\eqref{eq:U-W-mod}.
By a straightforward $\hbar$-adic analogue of
\cite[Theorem 2.21]{JKLT-G-phi-mod}, we get that the relations \eqref{tau1'}, \eqref{tau2'}, \eqref{tau3'}, \eqref{tau3.5'}, \eqref{tau4'} hold
with
\begin{align*}
  h_i(z)=Y_\E^\phi(h_{i,q}(z_1),z),\quad x_i^\pm(z)
  =Y_\E^\phi(x_{i,q}^\pm(z_1),z)\quad\te{for }i\in I.
\end{align*}
Combining Remark \ref{rem:x+x-} and \eqref{tau3'}, we get that the relations \eqref{eq:x+0x-} and \eqref{eq:x+1x-} hold with
\begin{align}\label{eq:assign}
  x_i^\pm=x_{i,q}^\pm(z),\quad h_i=h_{i,q}(z)\quad\te{for }i\in I.
\end{align}
It follows immediate from \eqref{phi5} that the relation \eqref{eq:serre} holds under the assignment \eqref{eq:assign}.
By using Proposition \ref{prop:universal-qaff}, we get an $\hbar$-adic nonlocal vertex algebra homomorphism
$\kappa_W:V_{\hbar,\tau}^\ell(\g)\to V_W$ determined by
\begin{align*}
  \kappa_W(h_i)=h_{i,q}(z),\quad \kappa_W(x_i^\pm)=x_{i,q}^\pm(z)\quad\te{for }i\in I.
\end{align*}
Note that for each $s\in\Z_N$, we have
\begin{align*}
  R(s)\kappa_W(a_i)=R(s)a_{i,q}(z)=a_{i,q}(\xi^{-s}z)
  =a_{\mu^s(i),q}(z)\quad\te{for }i\in I,\,a\in \{h,x^\pm\}.
\end{align*}
Since $V_W$ is generated by $a_i(z)$ for $i\in I$, $a\in \{h,x^\pm\}$, we have that
\begin{align*}
  R(s)\circ\kappa_W=\kappa_W\circ R(s)\quad\te{for }s\in\Z_N.
\end{align*}
Therefore, $W$ becomes a $(\Z_N,\chi_\phi)$-equivariant $\phi$-coordinated quasi $V_{\hbar,\tau}^\ell(\g)$-module
with module action $Y_W^\phi$ determined by
\begin{align*}
  Y_W^\phi(h_i,z)=h_{i,q}(z),\quad Y_W^\phi(x_i^\pm,z)=x_{i,q}^\pm(z)\quad\te{for }i\in I.
\end{align*}

On the other hand, let $(W,Y_W)$ be a $(\Z_N,\chi_\phi)$-equivariant $\phi$-coordinated quasi $V_{\hbar,\tau}^\ell(\g)$-module.
The relation \eqref{phi0} follows from the $(\Z_N,\chi_\phi)$-equivariant condition \eqref{eq:G-equiv}, the relation \eqref{phi5} follows from \eqref{eq:serre} under the assignment
\begin{align}\label{eq:assign-1}
  h_{i,q}(z)=Y_W^\phi(h_i,z),\quad x_{i,q}^\pm(z)=Y_W^\phi(x_i^\pm,z)\quad\te{for }i\in I.
\end{align}
By applying a straightforward $\hbar$-adic analogue of \cite[Proposition 2.15]{K-qaffva-rtu} to \eqref{tau1'},
\eqref{tau2'}, \eqref{tau3'}, \eqref{tau3.5'}, \eqref{tau4'}, we get that the relations \eqref{phi1},
\eqref{phi2}, \eqref{phi3}, \eqref{phi3.5} and \eqref{phi4} hold under the assignment \eqref{eq:assign-1}.
Therefore, $(W,Y_W^\phi(h_i,z),Y_W^\phi(x_i^\pm,z))$
becomes an object in $\mathcal R_\ell^\phi$.
\end{proof}

For each $i\in I$, we write
\begin{align*}
  \mathcal O(i)=\set{\mu^k(i)}{k\in\Z_N}\subset I
\end{align*}
for the orbit containing $i$, and $N_i$ the cardinality of $\mathcal O(i)$.
For $i,j\in I$, set
\begin{align*}
  \Gamma_{ij}=\set{k\in\Z_N}{a_{i,\mu^k(j)}\ne 0}\quad\te{and}\quad
  \Gamma_{ij}^\ast=\set{k\in\Z_N}{a_{i,\mu^k(j)}=a_{ij}}.
\end{align*}
Suppose that $\mu$ satisfies the following three linking conditions
\begin{align}
  \tag{LC1}\label{LC1}&\sum_{p\in\mathcal O(i)}a_{pi}>0\quad\te{for }i\in I,\\
  \tag{LC2}\label{LC2}&\Gamma_{ij}=\Gamma_{ij}^\ast\quad\te{for }i,j\in I\,\te{such that }i\not\in\mathcal O(j),\, a_{ij}<0,\\
  \tag{LC3}\label{LC3}&\Gamma_{ij}\,\,\te{is a subgroup of $\Z_N$ for }i,j\in I\,\te{such that }i\not\in\mathcal O(j),\,a_{ij}<0.
\end{align}
The condition (LC1) can be reformulated as follows: %

\begin{lem}\label{lem:linking}\cite[Sect.\,2.2]{FSS} The automorphism $\mu$ satisfies the condition (LC1) if and only if for every $i\in I$,
the Dynkin subdiagram of  $\mathcal O(i)$ is either

\te{(i)} a direct sum of type $A_1$, or

\te{(ii)} a direct sum of type $A_2$ with $a_{\mu^{N_i/2}(i),i}=-1$.
\end{lem}

For $i,j\in I$,  set
\begin{align}\label{caij}\check{a}_{ij}=s_i\sum_{p\in \mathcal O(i)}a_{pj}\in s_i\Z,\end{align}
where
\begin{equation}\label{si} s_i=\begin{cases} 1,\ \te{ if (i) holds in Lemma \ref{lem:linking}};\\
2,\ \te{ if (ii) holds in Lemma \ref{lem:linking}.}\end{cases}
\end{equation}
Define
\begin{align*}
  d_{ij}=|\set{k-l}{k,l\in \Gamma_{ij}}|,\quad d_i=N/N_i\quad\te{for }i,j\in I.
\end{align*}
For $i,j\in I$, we introduce the $\C[[\hbar]]$-valued polynomials:
\begin{align} \label{e:F}
&F^\pm_{ij}(z,w)=\prod_{k\in \Gamma_{ij}}\left( z-\xi^{k}q_i^{\pm a_{i\mu^k(j)}}w \right),\\ \label{e:G}
&G^\pm_{ij}(z,w)=\prod_{k\in \Gamma_{ij}}\left( q_i^{\pm a_{i\mu^k(j)}}z-\xi^{k}w \right),\\
\label{pi}
&p_i^\pm(z_1,z_2,z_3)=q_i^{\mp \frac{3}{2}d_{i}}z_1^{d_{i}}
        -(q_i^{\frac{d_{i}}{2}}+q_i^{-\frac{d_{i}}{2}})
            z_2^{d_{i}}
        +q_i^{\pm \frac{3}{2}d_{i}}
            z_3^{d_{i}},\quad\te{if}\ s_i=2,\\
            \label{pij}
&p_{ij}^\pm(z,w)=
    \left(
        z^{d_i}+q_i^{\mp d_i}w^{d_i}
    \right)^{s_i-1}
    \frac{
        q_i^{\pm 2d_{ij}}z^{d_{ij}}-w^{d_{ij}}
    }{
        q_i^{\pm 2d_i}z^{d_{i}}-w^{d_i}
    },\quad\te{if}\ \check{a}_{ij}<0,
\end{align}
and the formal series
\begin{align}
\label{e:g}
g_{ij}(z)
=\prod_{k\in \Z_N}\frac{ q_i^{a_{i\mu^k(j)}}-\xi^k z}{1-\xi^kq_i^{a_{i\mu^k(j)}}z},
\end{align}
which is expanded for $|z|<1$.
The following algebra was introduced in \cite{CJKT-qeala-II-twisted-qaffinization}.

\begin{de}\label{de:tqaffine}
The $\mu$-twisted quantum affinization $\U_\hbar(\wh\g_\mu)$
is the unital associative $\C[[\hbar]]$-algebra
generated by the set
\begin{align}\label{eq:tqagenerators}
\set{H_{i,n},\  X^\pm_{i,n},\ C}
{
   i\in I, n\in\Z
},
\end{align}
and subject to the relations in  terms of the generating functions
 \begin{align*}
 \Phi_i^\pm(z)=q^{\pm H_{i,0}}\, \te{exp}
    \left(
        \pm (q-q\inverse)\sum\limits_{\pm m> 0}H_{i,m}z^{-m}
    \right),\quad X^\pm_i(z)=\sum\limits_{m\in \Z} X^\pm_{i,m}z^{-m}.
\end{align*}
The relations are ($i,j\in I$):
\begin{align}
&\tag{Q0}\label{Q0} \Phi^\pm_{\mu(i)}(z)=\Phi^\pm_i(\xi\inverse z),\quad
X^\pm_{\mu(i)}(z)=X^\pm_i(\xi\inverse z),\\
&\tag{Q1}\label{Q1} [C,\Phi_i^\pm(z)]=0=[C,X_i^\pm(z)]=[\Phi_i^\pm(z),\Phi_j^\pm(w)],\\
&\tag{Q2}\label{Q2} \Phi^+_i(z)\Phi^-_j(w)=\Phi^-_j(w)\Phi^+_i(z)
    g_{ij}(q^{rC} w/z)\inverse g_{ij}(q^{-rC}w/z),\\
&\tag{Q3}\label{Q3}
\Phi^+_i(z)x^\pm_j(w)=X^\pm_j(w)\Phi^+_i(z)
    g_{ij}(q^{\mp \half rC}w/z)^{\pm 1},\\
&\tag{Q4}\label{Q4}
\Phi^-_i(z)X^\pm_j(w)=X^\pm_j(w)\Phi^-_i(z)
    g_{ji}(q^{\mp \half rC}z/w)^{\mp 1},\\
&\tag{Q5}\label{Q5}[X_i^+(z),X_j^-(w)]
=\frac{1}{q_i-q_i\inverse}
    \sum_{k\in\Z_N}\delta_{i,\mu^k(j)}\\
 &\nonumber  \quad\times \Bigg(
        \phi_i^+(z q^{-\half rC})\delta\left(
            \frac{\xi^kw q^{rC}}{z}
        \right)
        -
        \phi_i^-(z q^{\half rC})\delta\left(
            \frac{\xi^kw q^{-rC} }{z}
        \right)
    \Bigg),\\
&\tag{Q6}\label{Q6}
F^\pm_{ij}(z,w)X^\pm_i(z)X^\pm_j(w)=
    G^\pm_{ ij}(z,w)X^\pm_j(w)X^\pm_i(z),\\
&\tag{Q7}\label{Q7}
[X_i^\pm(z_1),X_j^\pm(z_2)]=0\quad\te{if }\check a_{ij}=0,\\
&\tag{Q8}\label{Q8}\sum_{\sigma\in S_{3}}
    p_i^\pm(z_{\sigma(1)},z_{\sigma(2)},z_{\sigma(3)})\,
                X_i^\pm(z_{\sigma(1)})X_i^\pm(z_{\sigma(2)})X_i^\pm(z_{\sigma(3)})=0,
\quad
    \te{if}\ \ s_i=2,\\
&\tag{Q9}\label{Q9}
\sum_{\sigma\in S_{{m}_{ij}}}\sum_{r=0}^{{m}_{ij}}
    (-1)^r \binom{{m}_{ij}}{r}_{q_i^{d_{ij}}}
    \prod_{1\le a<b\le m_{ij}}p_{ij}^\pm(z_{\sigma(a)},z_{\sigma(b)})
    X_i^\pm(z_{\sigma(1)})\cdots X_i^\pm(z_{\sigma(r)})\\
 &\qquad\nonumber  \cdot X_j^\pm(w)
       X_i^\pm(z_{\sigma(r+1)})\cdots X_i^\pm(z_{\sigma({m}_{ij})})=0,\quad
    \te{if}\ \ \check{a}_{ij}<0.
\end{align}
\end{de}

Let $\U_\hbar^f(\wh\g_\mu)$ be the unital associative $\C[[\hbar]]$-algebra generated by \eqref{eq:tqagenerators} subject to relations \eqref{Q0}-\eqref{Q4}, \eqref{Q6} and
\begin{align}
  \tag{Q5.5}\label{Q5.5}
  \prod_{s\in\Z_N}
  (1-q^{rC}\xi^sz_2/z_1)^{\delta_{i,\mu^s(j)}}
  (1-q^{-rC}\xi^sz_2/z_1)^{\delta_{i,\mu^s(j)}}
  [X_i^+(z_1),X_j^-(z_2)]=0.
\end{align}

\begin{de}
A $\U_\hbar^f(\wh\g_\mu)$-module $W$ is called \emph{restricted} if $W$ is topologically free as a $\C[[\hbar]]$-module and
\begin{align*}
  \Phi_i^\pm(z),\quad X_i^\pm(z)\in \E_\hbar(W)\quad\te{for }i\in I,
\end{align*}
and is called of \emph{level $\ell$} for some $\ell\in \C$, if $C=\ell$ on $W$.
\end{de}

For $i,j\in I$, set
\begin{align*}
  f_{ij}^\pm(z_1,z_2)=F_{ij}^\pm(z_1,z_2)\prod_{s\in\Z_N:a_{i,\mu^s(j)}>0}
  (z_1-\xi^sz_2)\inv.
\end{align*}
Recall the normal ordered products introduced in \cite[Definition 5.7]{CJKT-qeala-II-twisted-qaffinization}.
\begin{de}
Let $W$ be a restricted $\U_\hbar^f(\wh\g_\mu)$-module
and let $i_1,\dots,i_m\in I$. Define
\begin{align*}
  &\:X_{i_1}^\pm(z_1)X_{i_2}^\pm(z_2)\cdots X_{i_m}^\pm(z_m)\;\\
  =&\prod_{1\le a<b\le m}\iota_{z_a,z_b}f_{i_a,i_b}^\pm(z_a,z_b)
  X_{i_1}^\pm(z_1)X_{i_2}^\pm(z_2)\cdots X_{i_m}^\pm(z_m).
\end{align*}
\end{de}

The following result was given in \cite[Proposition 5.8]{CJKT-qeala-II-twisted-qaffinization}.
\begin{prop}
For a restricted $\U_\hbar^f(\wh\g_\mu)$-module $W$ and $i_1,\dots, i_m\in I$, we have
\begin{align*}
  \:X_{i_1}^\pm(z_1)X_{i_2}^\pm(z_2)\cdots X_{i_m}^\pm(z_m)\;\in \E_\hbar^{(m)}(W),
\end{align*}
and for $\sigma\in S_m$, we have
\begin{align*}
  &\:X_{i_{\sigma(1)}}^\pm(z_{\sigma(1)})
  X_{i_{\sigma(2)}}^\pm(z_{\sigma(2)})\cdots X_{i_{\sigma(m)}}^\pm(z_{\sigma(m)})\;\\
  =&\bigg(\prod_{\substack{1\le a<b\le m\\ \sigma(a)>\sigma(b)}}C_{i_a,i_b}\bigg)
  \:X_{i_1}^\pm(z_1)X_{i_2}^\pm(z_2)\cdots X_{i_m}^\pm(z_m)\;,
\end{align*}
where
\begin{align*}
  C_{ij}=\prod_{s\in\Z_N:a_{i,\mu^s(j)}<0}(-\xi^s)\quad
  \te{for }i,j\in I.
\end{align*}
\end{prop}

The following result was proved in
\cite[Theorem 5.17]{CJKT-qeala-II-twisted-qaffinization},
which rewrite the relations \eqref{Q8} and \eqref{Q9}.
\begin{thm}\label{thm:another-pre}
Let $W$ be a restricted $\U_\hbar^f(\wh\g_\mu)$-module.
Then $W$ is a $\U_\hbar(\wh\g_\mu)$-module if and only if the relation \eqref{Q5} and the following relation hold
\begin{align}
\tag{Q10}\label{Q10}
  \:X_i^\pm(q_i^{a_{ij}}z)X_i^\pm(q_i^{a_{ij}+2}z)
  \cdots X_i^\pm(q_i^{-a_{ij}}z)X_j^\pm(z)\;=0\quad\te{for }i,j\in I\,\te{with }a_{ij}<0.
\end{align}
\end{thm}

A straightforward calculation proves the
following analogue of \cite[Lemma 8.5]{K-Quantum-aff-va}.
\begin{lem}\label{lem:another-pre}
Let $\ell\in \C$ and let $W$ be a restricted $\U_\hbar^f(\wh\g_\mu)$-module of level $\ell$.
Define $\wh X_i^+(z)=X_i^+(z)$ and
\begin{align*}
  \wh\Phi_i(z)=F\left(z\pd z\right)\inv \log\frac{\Phi_i^+(zq^{\half r\ell})}{\Phi_i^-(z q^{-\half r\ell})},\quad
  \wh X_i^-(z)=X_i^-(zq^{-r\ell})\Phi_i^+(zq^{-\half r\ell})\inv.
\end{align*}
Then $(W,\wh \Phi_i(z),\wh X_i^\pm(z))$ becomes an object of $\mathcal M_\ell^\phi$.

On the other hand, let $(W,h_{i,q}(z),x_{i,q}^\pm(z))
\in \obj\mathcal M_\ell^\phi$.
Write
\begin{align*}
  h_{i,q}^\pm(z)=\sum_{\pm n>0}h_{i,q}(n)z^{-n}+\half h_{i,q}(0),\quad\te{where }h_{i,q}(z)=\sum_{n\in\Z}h_{i,q}(n)z^{-n}.
\end{align*}
Define $\check x_{i,q}^+(z)=x_{i,q}^+(z)$ and
\begin{align*}
  \check h_{i,q}^\pm(z)=\exp\left(\pm F\left(z\pd z\right)h_{i,q}^\pm(zq^{\mp \half r\ell})\right),\quad
  \check x_{i,q}^-(z)=x_{i,q}^-(zq^{r\ell})\check h_{i,q}^+(zq^{\half r\ell}).
\end{align*}
Then $W$ becomes a restricted $\U_\hbar^f(\wh\g_\mu)$-module of level $\ell$ with module action
\begin{align*}
  \Phi_i^\pm(z)=\check h_{i,q}^\pm(z),\quad X_i^\pm(z)=\check x_{i,q}^\pm(z)\quad\te{for }i\in I.
\end{align*}
\end{lem}

The following is the main result of this section.

\begin{thm}
The category of restricted $\U_\hbar(\wh\g_\mu)$-modules of level $\ell$ is isomorphic to
the category of $(\Z_N,\chi_\phi)$-equivariant $\phi$-coordinated quasi $V_{\hbar,\tau}^\ell(\g)$-modules.
\end{thm}

\begin{proof}
By Proposition \ref{prop:equiv-mods=R}, it suffices to show that the category of restricted $\U_\hbar(\wh\g_\mu)$-modules of level $\ell$ is isomorphic to $\mathcal R_\ell^\phi$.
Lemma \ref{lem:another-pre} implies that a topologically free $\C[[\hbar]]$-module is a restricted $\U_\hbar^f(\wh\g)$-module of level $\ell$ if and only if it belongs to $\mathcal M_\ell^\phi$.
Let $W$ be a restricted $\U_\hbar^f(\wh\g_\mu)$-module of level $\ell$ and let $(W,\wh \Phi_i(z),\wh X_i^\pm(z))$ be the corresponding object of $\mathcal M_\ell^\phi$.
By using Theorem \ref{thm:another-pre}, the theorem reduces to prove that relations \eqref{Q5} and \eqref{Q10} are equivalent to \eqref{phi3} and \eqref{phi5}.
Using Lemma \ref{lem:U-W}, we obtain a $\Z_N$-module
$\hbar$-adic nonlocal vertex algebra $V_W$.
An argument similar to the proof of \cite[Lemma 8.7]{K-Quantum-aff-va}
shows that
\begin{align*}
  E_\tau\big(\wh \Phi_i(z)\big)=\Phi_i^-(zq^{-\frac 32 r\ell})\Phi_i^+(zq^{-\half r\ell})\inv\quad\te{for }i\in I.
\end{align*}
By using this, a straightforward calculation shows that relation \eqref{Q5} is equivalent to \eqref{phi3}.
A straightforward calculation yields
\begin{align*}
  f_{ij}^+(z_1,z_2)\wh X_i^-(z_1)\wh X_j^-(z_2)
  =C_{ij}f_{ji}^+(z_2,z_1)\wh X_j^-(z_2)\wh X_i^-(z_1)
  \quad\te{for }i,j\in I.
\end{align*}
We also define the normal ordered product of the currents $\wh X_i^\pm(z)$ by
\begin{align*}
  &\:\wh X_{i_1}^\pm(z_1)\wh X_{i_2}^\pm(z_2)
  \cdots \wh X_{i_m}^\pm(z_m)\;\\
  =&\bigg(\prod_{1\le a<b\le m}\iota_{z_a,z_b} f_{i_a,i_b}^+(z_a,z_b)\bigg)
  \wh X_{i_1}^\pm(z_1)\cdots \wh X_{i_m}^\pm(z_m)
  \quad\te{for }i_1,\dots,i_m\in I.
\end{align*}
It is straightforward to verify that the relation \eqref{Q10} is equivalent to the following relation
for $i,j\in I$ with $a_{ij}<0$:
\begin{align}\tag{Q10'}\label{Q10'}
  &\:\wh X_i^\pm(q_i^{a_{ij}}z)\wh X_i^\pm(q_i^{a_{ij}+2}z)
  \cdots \wh X_i^\pm(q_i^{-a_{ij}}z)\wh X_j^\pm(z)\;=0.
\end{align}
By an argument similar to the proof of \cite[Lemma 8.8]{CJKT-qeala-II-twisted-qaffinization}, we get that \eqref{Q10'} is equivalent to the relation
\begin{align*}
  \left(\left(\wh X_i^\pm(z)\right)_0^\phi\right)^{m_{ij}}\wh X_j^\pm(z)=0\quad\te{for }i,j\in I\,\te{with }a_{ij}<0,
\end{align*}
which is exactly \eqref{phi5}.
\end{proof}

\begin{appendices}

\section{Related to double Yangians}\label{app:YD}

In this section, we show that $V_{\hbar,\varepsilon}^\ell(\g)$ is isomorphic to the $\hbar$-adic weak quantum vertex algebra $\mathcal V_A(r\ell)$ defined in \cite[Section 5.2]{KL-YD-1}, corresponding to the double Yangian $\wh{\mathcal{DY}}(A)$ (see \cite[Definition 5.1]{KL-YD-1} for the definition).
Recall the following notion given in \cite[Definition 4.4]{KL-YD-1}.
\begin{de}
Let $W$ be a $\C[[\hbar]]$-module and let $a(z),b(z)\in \End(W)[[z,z\inv]]$ and $m\in\Z_+$. The relation
\begin{align}
  \sum_{\sigma\in S_m}\sum_{r=0}^m\binom{m}{r}(-1)^r a(z_{\sigma(1)})a(z_{\sigma(2)})\cdots a(z_{\sigma(r)})b(w) a(z_{\sigma(r+1)})
\cdots a(z_{\sigma(m)})=0
\end{align}
is referred to as the {\em Serre-like relation for $(a(z), b(z),m)$}.
\end{de}

Let $\wt \U$ be the free associative algebra with identity over $\C$ generated by set
\begin{align*}
  \{\wt \kappa\}\cup\set{\wt e_i(m),\wt f_i(m),\wt h_i(m)}{i\in I,m\in\Z},
\end{align*}
and let $\wt J_{A,\ell}^1$ be the two-sided ideal generated by element $\wt \kappa-\ell$, the relations
\begin{align}
&[\wt{h}_i(z_1),\wt{h}_j(z_2)]
=[r_ia_{i,j}]_{q^{\pd{z_2}}}
[\ell]_{q^{\pd{z_2}}}((z_1-z_2+\ell\hbar)^{-2}-
(z_2-z_1+\ell\hbar)^{-2}),\label{h-i-h-j}\\
&[\wt{h}_i(z_1),\tilde{e}_j(z_2)]
=\wt{e}_j(w)[r_ia_{i,j}]_{q^{\pd{z_2}}}
((z_1-z_2+\ell\hbar)^{-1}-(z_2-z_1+\ell\hbar)^{-1}),\label{h-i-e-j}\\
&[\wt{h}_i(z_1),\tilde{f}_j(z_2)]
=-\tilde{f}_j(z_2)[r_ia_{i,j}]_{q^{\pd{z_2}}}
((z_1-z_2+\ell\hbar)^{-1}-(z_2-z_1+\ell\hbar)^{-1}),\label{h-i-f-j}\\
&(z_1-z_2-r_ia_{ij}\hbar)\wt{e}_i(z_1)\wt{e}_j(z_2)
=(z_1-z_2+r_ia_{ij}\hbar)\wt{e}_j(z_2)\wt{e}_i(z_1),\label{e-i-e-j}\\
&(z_1-z_2-r_ia_{ij}\hbar)\wt{f}_i(z_1)\wt{f}_j(z_2)
=(z_1-z_2+r_ia_{ij}\hbar)\wt{f}_j(z_2)\wt{f}_i(z_1),\label{f-i-f-j}\\
&\nonumber
(z_1-z_2+r_ia_{ij}\hbar) \wt{e}_i(z_1)\wt{f}_j(z_2)=(z_1-z_2-r_ia_{ij}\hbar) \wt{f}_j(z_2)\wt{e}_i(z_1)\quad \text{ if }i\ne j,\\
&\nonumber
(z_1-z_2)(z_1-z_2+2\ell\hbar)(z_1-z_2+2r_i\hbar) \wt{e}_i(z_1)\wt{f}_i(z_2)\\
&\qquad\nonumber=(z_1-z_2)(z_1-z_2+2\ell\hbar)
(z_1-z_2-2r_i\hbar) \wt{f}_i(z_2)\wt{e}_i(z_1).
\end{align}
Define $\U_{A,\ell}'$ to be the quotient algebra of $\wt \U[[\hbar]]$ modulo $\overline{[\wt J_{A,\ell}^1]}$.
Let $\wt J_{A,\ell}^2$ be the two-sided ideal of $\U_{A,\ell}'$ generated by the Serre-like relations for $(\wt e_i(z),\wt e_j(z),1-a_{ij})$ and $(\wt f_i(z),\wt f_j(z),1-a_{ij})$ for $i,j\in I$ with $a_{ij}\le 0$.
Define $\U_{A,\ell}$ to be the quotient algebra of $\U_{A,\ell}'$ modulo $\overline{[\wt J_{A,\ell}^2]}$.

For $i\in I$ and $m\in\Z$, we set
\begin{align*}
  \bar e_i(m)=\wt e_i(m)+\overline{[\wt J_{A,\ell}]},\quad
  \bar f_i(m)=\wt f_i(m)+\overline{[\wt J_{A,\ell}]},\quad
  \bar h_i(m)=\wt h_i(m)+\overline{[\wt J_{A,\ell}]}\in\U_{A,\ell}.
\end{align*}
Let $\mathcal J_{A,\ell}$ be the left ideal of $\U_{A,\ell}$ generated by the subset
\begin{align*}
  \set{\bar e_i(n),\,\bar f_i(n),\,\bar h_i(n)}{i\in I,\,n\in\N}.
\end{align*}
Define
\begin{align*}
  \mathcal V_{A,\ell}=\U_{A,\ell}/\overline{[\mathcal J_{A,\ell}]}.
\end{align*}
Set $\vac=1+\overline{[\mathcal J_{A,\ell}]}\in \mathcal V_{A,\ell}$.
Set
\begin{align*}
  \wh e_i=\bar e_i(-1)\vac,\quad \wh f_i=\bar f_i(-1)\vac,\quad
  \wh h_i=\bar h_i(-1)\vac\in \mathcal V_{A,\ell}\quad\te{for }i\in I,
\end{align*}
and set
\begin{align*}
  \bar e_i(z)=\sum_{n\in\Z}\bar e_i(n)z^{-n-1},\quad
  \bar f_i(z)=\sum_{n\in\Z}\bar f_i(n)z^{-n-1},\quad
  \bar h_i(z)=\sum_{n\in\Z}\bar h_i(n)z^{-n-1}.
\end{align*}
The following result was given in \cite[Proposition 5.10]{KL-YD-1}.
\begin{prop}\label{prop:wqva}
There exists an $\hbar$-adic weak quantum vertex algebra structure on $\mathcal V_{A,\ell}$, which is uniquely determined by the condition that $\vac$ is the vacuum vector and
\begin{align*}
  Y(\wh e_i,z)=\bar e_i(z),\quad Y(\wh f_i,z)=\bar f_i(z),\quad Y(\wh h_i,z)=\bar h_i(z)\quad\te{for }i\in I.
\end{align*}
\end{prop}

By using \cite[Theorem 4.13]{KL-YD-1}, we have that the relations \eqref{e-i-e-j} and \eqref{f-i-f-j} are equivalent to the following relations respectively:
\begin{align*}
  (\bar e_i(z))_0^{1-a_{ij}}\bar e_j(z)=0\quad\te{and}\quad
  (\bar f_i(z))_0^{1-a_{ij}}\bar f_j(z)=0\quad\te{for }i,j\in I,\,\te{with }a_{ij}\le 0,
\end{align*}
which are equivalent to the following relations respectively:
\begin{align}\label{serre}
  (\wh e_i)_0^{1-a_{ij}}\wh e_j=0\quad\te{and}\quad
  (\wh f_i)_0^{1-a_{ij}}\wh f_j=0\quad\te{for }i,j\in I,\,\te{with }a_{ij}\le 0.
\end{align}

For $i\in I$, set
\begin{align*}
C_i=e^{(1-\ell)\hbar \partial}E^{-}(-\wh{h}_i,-2\hbar){\bf 1}\in \mathcal{V}_{A,\ell},
\end{align*}
where $\partial$ is the canonical derivation of $\mathcal V_{A,\ell}$ and
$$ E^{-}(-\wh{h}_i,-2\hbar)
=\exp\left(\sum_{n=1}^\infty\frac{1}{n}\wh{h}_i(-n)(-2\hbar)^n\right).$$
We note that
\begin{align}\label{C-i}
  C_i=\exp\left( \left(-q^{-\ell\partial}F(\partial)\wh h_i\right)_{-1} \right)\vac.
\end{align}
Denote by $K_{A,\ell}$ the ideal generated by the following vectors
\begin{align}\label{eq:e-i-f-j-sing}
  (\wh e_i)_0(\wh f_j)-\frac{\delta_{ij}}{2r_i\hbar}(1-C_i),\quad
  (\wh e_i)_n(\wh f_j)-\frac{\delta_{ij}}{r_i}(-2\ell\hbar)^{n-1}C_i\quad\te{for }i,j\in I,\,n\ge 1.
\end{align}
Define $\mathcal V_A(\ell)$ to be the $\hbar$-adic weak quantum vertex quotient algebra of $\mathcal V_{A,\ell}$ modulo $\overline{[K_{A,\ell}]}$.
For $i\in I$, set
\begin{align*}
  \check e_i=\wh e_i+\overline{[K_{A,\ell}]},\quad
  \check f_i=\wh f_i+\overline{[K_{A,\ell}]},\quad
  \check h_i=\wh h_i+\overline{[K_{A,\ell}]}\in \mathcal V_A(\ell).
\end{align*}
The following relation was given in \cite[Remark 5.14]{KL-YD-1}:
\begin{align}\label{eq:e-i-f-j-total}
  &Y(\check e_i,z_1)Y(\check f_j,z_2)-\frac{z_2-z_1+r_ia_{ij}\hbar}{z_2-z_1-r_ia_{ij}\hbar}
  Y(\check f_j,z_2)Y(\check e_i,z_1)\\
  =&\frac{\delta_{ij}}{2r_i\hbar}
  \left(
  z_1\inv\delta\left(\frac{z_2}{z_1}\right)
  -Y(q^{(1-\ell)\partial}E^-(-\check h_i,-2\hbar)\vac,z_2)z_1\inv\left(\frac{z_2-2\ell\hbar}{z_1}\right)
  \right),\nonumber
\end{align}
which implies the following relation
\begin{align}\label{e-i-f-j}
  &(z_1-z_2)^{\delta_{ij}}(z_1-z_2+2\ell\hbar)^{\delta_{ij}}
  \Big( Y(\check e_i,z_1)Y(\check f_j,z_2)\\
  &\quad-\iota_{z_2,z_1}
  \frac{z_2-z_1+r_ia_{ij}\hbar}{z_2-z_1-r_ia_{ij}\hbar}
  Y(\check f_j,z_2)Y(\check e_i,z_1) \Big)=0.\nonumber
\end{align}

The following is the main result of this section.
\begin{prop}
There is an $\hbar$-adic weak quantum vertex algebra isomorphism $\mathcal V_A(r\ell)\to V_{\hbar,\varepsilon}^\ell(\g)$ determined by
\begin{align}
  \check h_i\mapsto h_i,\quad \check e_i\mapsto x_i^+,\quad \check f_i\mapsto x_i^-\quad\te{for }i\in I.
\end{align}
\end{prop}

\begin{proof}
Applying Proposition \ref{prop:universal-qaff} to relations \eqref{h-i-h-j}, \eqref{h-i-e-j}, \eqref{h-i-f-j}, \eqref{e-i-e-j}, \eqref{f-i-f-j},
\eqref{e-i-f-j}, \eqref{eq:e-i-f-j-sing},  \eqref{eq:serre} and \eqref{C-i}, we get an $\hbar$-adic nonlocal vertex algebra homomorphism $\varphi:V_{\hbar,\varepsilon}^\ell(\g)\to \mathcal V_A(r\ell)$ determined by
\begin{align*}
  \varphi(h_i)=\check h_i,\quad \varphi(x_i^+)=\check e_i,\quad \varphi(x_i^-)=\check f_i\quad\te{for }i\in I.
\end{align*}

On the other hand,
let $\mathcal M$ be the category, whose objects are topologically free $\C[[\hbar]]$-modules $W$ equipped with fields
\begin{align*}
  h_i(z)=\sum_{n\in\Z}h_i(n)z^{-n-1},\quad x_i^\pm(z)=\sum_{n\in\Z}x_i^\pm(n)z^{-n-1}\in \E_\hbar(W)\quad\te{for }i\in I,
\end{align*}
satisfies the relations \eqref{tau1}, \eqref{tau2}, \eqref{tau4} and
\begin{align}
  &\label{tau3.75}(z_1-z_2)^{\delta_{ij}}
  (z_1-z_2+2r\ell\hbar)^{\delta_{ij}}
  (z_1-z_2+r_ia_{ij}\hbar)x_i^+(z_1)x_j^-(z_2)\\
  =&(z_1-z_2)^{\delta_{ij}}
  (z_1-z_2+2r\ell\hbar)^{\delta_{ij}}
  (z_1-z_2-r_ia_{ij}\hbar)x_j^-(z_2)x_i^+(z_1).\nonumber
\end{align}
Let $R$ be the ideal of $V(\mathcal M)$ generated by
\eqref{eq:serre} and
\begin{align*}
  &(x_i^+)_0(x_j^-)-\frac{\delta_{ij}}{2r_i\hbar}
  (\vac-E_\varepsilon(h_i)),(x_i^+)_n(x_j^-)-\frac{\delta_{ij}}{r_i}(-2r\ell\hbar)^{n-1}E_\varepsilon(h_i)
  \quad\te{for }i,j\in I,\,n\ge 1.
\end{align*}
Then there is an $\hbar$-adic nonlocal vertex algebra isomorphism $V_{\hbar,\varepsilon}^\ell(\g)\to V(\mathcal M)/\overline{[R]}$ determined by
\begin{align*}
  h_i\mapsto h_i,\quad x_i^\pm\mapsto x_i^\pm\quad\te{for }i\in I.
\end{align*}
Remark \ref{rem:direct-construct} yields a $\U_{A,\ell}'$-module structure on $V(\mathcal M)$ defined by ($i\in I$)
\begin{align*}
  \wt h_i(z)=h_i(z)=Y(h_i,z),\quad \wt e_i(z)=x_i^+(z)=Y(x_i^+,z),\quad \wt f_i(z)=x_i^-(z)=Y(x_i^-,z).
\end{align*}
It then induces a $\U_{A,\ell}'$-module structure on $V_{\hbar,\varepsilon}^\ell(\g)$.
Note that the relation \eqref{eq:serre} implies
that
\begin{align*}
  Y(x_i^\pm,z)_0^{m_{ij}}Y(x_j^\pm,z)=Y\left((x_i^\pm)_0^{m_{ij}}x_j^\pm,z\right)=0.
\end{align*}
By using \cite[Theorem 4.13]{KL-YD-1}, we have that the Serre-like relation $(Y(x_i^\pm,z),Y(x_j^\pm,z),1-a_{ij})$ holds for any $i,j\in I$ with $a_{ij}\le 0$.
Then the $\U_{A,\ell}'$-module structure on $V_{\hbar,\varepsilon}^\ell(\g)$ induces a $\U_{A,\ell}$-module structure.
The vacuum property shows that
\begin{align*}
  \bar h_i(z)\vac=Y(h_i,z)\vac,\quad \bar e_i(z)\vac=Y(x_i^+,z)\vac,\quad \bar f_i(z)=Y(x_i^-,z)\vac\in V_{\hbar,\varepsilon}^\ell(\g)[[z]].
\end{align*}
The definition of $\mathcal V_{A,\ell}$ induces a $\U_{A,\ell}$-modulo homomorphism $\pi:\mathcal V_{A,\ell}\to V_{\hbar,\varepsilon}^\ell(\g)$ such that
$\pi(\vac)=\vac$.
Proposition \ref{prop:wqva} shows that
\begin{align*}
  &\pi(\wh h_i)=h_i,\quad \pi(\wh e_i)=x_i^+,\quad \pi(\wh f_i)=x_i^-,\quad
  \pi(Y(\wh h_i,z)v)=Y(h_i,z)\pi(v),\\
  &\pi(Y(\wh e_i,z)v)=Y(x_i^+,z)\pi(v),\quad
  \pi(Y(\wh f_i,z)v)=Y(x_i^-,z)\pi(v)
\end{align*}
for $i\in I$ and $v\in \mathcal V_{A,\ell}$.
Hence, $\pi$ is also an $\hbar$-adic nonlocal vertex algebra homomorphism.
Finally, by comparing \eqref{C-i}, \eqref{eq:e-i-f-j-sing} with \eqref{eq:def-E}, \eqref{eq:x+0x-}
and \eqref{eq:x+1x-}, we get an $\hbar$-adic nonlocal vertex algebra homomorphism $\psi:\mathcal V_A(r\ell)\to V_{\hbar,\varepsilon}^\ell(\g)$ induces from $\pi$.
It is easy to verify that $\varphi$ and $\psi$ are inverse of each other, which completes the proof of proposition.
\end{proof}

\section{Related to quantum affine algebras}\label{app:qaff}

Let $\mu$ be the identity automorphism in Section \ref{sec:tqaff}.
Then $\tau$ specializes to the following relations
\begin{align*}
  &\tau_{ij}^{0,0}(z)=
  q_i^{-a_{ij}\pd z}
  [r\ell]_{q^{\pd z}}^2\pdiff{z}{2}f_0(z),\\
  &\tau_{ij}^{\bullet,0}(z)=\tau_{ij}^{0,\bullet}(z)
  =-
  q_i^{-a_{ij}\pd z}[r\ell]_{q^{\pd z}}\pd zf_0(z),\\
  &\tau_{ij}^{\bullet,\bullet}(z)
  =\exp\left(f_0(z-r_ia_{ij}\hbar)\right).
\end{align*}
In this section, we show that the corresponding quantum affine vertex algebras $V_{\hbar,\tau}^\ell(\g)$ and $L_{\hbar,\tau}^\ell(\g)$ are isomorphic to that introduced in \cite{K-Quantum-aff-va}.

We first recall the definition of quantum affine vertex algebras introduced in \cite{K-Quantum-aff-va}.
Let $\mathcal M$ be the category, whose objects are topologically free $\C[[\hbar]]$-modules $W$ equipped with fields
\begin{align*}
  h_{i,\hbar}(z),\quad x_{i,\hbar}^\pm(z)\in \E_\hbar(W)
  \quad\te{for }i\in I,
\end{align*}
satisfying the following relations (see \cite[(6.6), (6.7), (6.8) and (6.9)]{K-Quantum-aff-va}):
\begin{align}
  &[h_{i,\hbar}(z_1),h_{j,\hbar}(z_2)]
  =[r_ia_{ij}]_{q^{\pd{z_2}}}[r\ell]_{q^{\pd{z_2}}}\label{eq:local-h-1}\\
  &\quad\times
  \left(\iota_{z_1,z_2}q^{-r\ell\pd{z_2}}-\iota_{z_2,z_1}q^{r\ell\pd{z_2}}\right)\frac{e^{-z_1+z_2}}{(1-e^{-z_1+z_2})^2},\nonumber\\
  &[h_{i,\hbar}(z_1),x_{j,\hbar}^\pm(z_2)]
  =\pm x_{j,\hbar}^\pm(z_2)[r_ia_{ij}]_{q^{\pd{z_2}}}\label{eq:local-h-2}\\
  &\quad\times
  \left(\iota_{z_1,z_2}q^{-r\ell\pd{z_2}}-\iota_{z_2,z_1}q^{r\ell\pd{z_2}}\right)\frac{1+e^{-z_1+z_2}}{2-2e^{-z_1+z_2}},\nonumber\\
  &\iota_{z_1,z_2}\frac{1-q_i^{a_{ij}}e^{-z_1+z_2}}{(1-e^{-z_1+z_2})^{\delta_{ij}}} x_{i,\hbar}^\pm(z_1)x_{j,\hbar}^\pm(z_2)\label{eq:local-h-3-pre}\\
  &\quad  =\iota_{z_2,z_1}\frac{q_i^{a_{ij}}-e^{-z_1+z_2}}{(1-e^{-z_1+z_2})^{\delta_{ij}}} x_{j,\hbar}^\pm(z_2)x_{i,\hbar}^\pm(z_1),\nonumber\\
  &(z_1-z_2)^{\delta_{ij}}(z_1-z_2+2r\ell\hbar)^{\delta_{ij}}\label{eq:local-h-4}\\
  &\quad\times\left(x_{i,\hbar}^+(z_1)x_{j,\hbar}^-(z_2)
    -\frac{1-q_i^{a_{ij}}e^{z_2-z_1}}{q_i^{a_{ij}}-e^{z_2-z_1}}
  x_{j,\hbar}^-(z_2)x_{i,\hbar}^+(z_1)\right)=0.\nonumber
\end{align}
Let $A$ be the GCM of $\g$, and let $\ell\in \C$.
By using \cite[Proposition 5.3]{K-Quantum-aff-va}, one obtained an $\hbar$-adic nonlocal vertex algebra $F_\tau(A,\ell)$.
The quantum affine vertex algebra $V_{\hat\g,\hbar}(\ell,0)$ was defined to be the quotient
$\hbar$-adic nonlocal vertex algebra of $F_\tau(A,\ell)$ modulo the minimal closed ideal that is closed under $[\cdot]$ (see \eqref{eq:Husdorff}) and contains the following elements
\begin{align}
  &\left(x_{i,\hbar}^+\right)_0x_{i,\hbar}^--(q_i-q_i\inv)\inv
    \left(\vac-E(h_{i,\hbar})\right)\quad\te{for } i\in I,\label{eq:x+0x-'}\\
  &\left(x_{i,\hbar}^+\right)_1x_{i,\hbar}^-+2r\ell\hbar(q_i-q_i\inv)\inv
  E(h_{i,\hbar})\quad\te{for } i\in I,\label{eq:x+1x-'}\\
  &\left(x_{i,\hbar}^\pm\right)_0^{m_{ij}}x_{j,\hbar}^\pm\quad\te{for } i,j\in I\,\te{with }a_{ij}\le 0,\label{eq:serre'}
\end{align}
where
\begin{align}\label{eq:def-E'}
  E(h_{i,\hbar})=\left(\frac{F(r_i+r\ell)}{F(r_i-r\ell)}\right)^\half
    \exp\left(\left(-q^{-r\ell\partial} F\left(\partial\right)h_{i,\hbar}\right)_{-1}\right)\vac.
\end{align}
If $\ell\in\Z_+$, the quantum affine vertex algebra  $L_{\hat\g,\hbar}(\ell,0)$ was defined to be the  quotient
$\hbar$-adic nonlocal vertex algebra of $V_{\hat\g,\hbar}(\ell,0)$ modulo the minimal closed ideal that is closed under $[\cdot]$ (see \eqref{eq:Husdorff}) and contains the following elements
\begin{align}
  & \left(x_{i,\hbar}^\pm\right)_{-1}^{r\ell/r_i}x_{i,\hbar}^\pm\quad\te{for } i\in I.\label{eq:def-integrable'}
\end{align}

\begin{prop}
Let $X=V$ or $L$.
There is an $\hbar$-adic quantum vertex algebra isomorphism from $X_{\hbar,\tau}^\ell(\g)$ to $X_{\hat\g,\hbar}(\ell,0)$ determined by
\begin{align*}
  h_i\mapsto h_{i,\hbar},\quad x_i^+\mapsto \frac{q_i-q_i\inv}{2r_i\hbar} x_{i,\hbar}^+,\quad
  x_i^-\mapsto x_i^-\quad\te{for }i\in I.
\end{align*}
\end{prop}

\begin{proof}
Note that the relations \eqref{tau1'}, \eqref{tau2'}, \eqref{tau3.5'} are exactly the relations \eqref{eq:local-h-1}, \eqref{eq:local-h-2} and \eqref{eq:local-h-4}, respectively.
In addition, the relation \eqref{eq:local-h-3-pre} implies the relation \eqref{tau4'}.
On the other hand, a similar argument to the proof of \cite[Lemma 5.6]{CJKT-qeala-II-twisted-qaffinization} shows that the relation \eqref{tau4'} also implies the relation \eqref{eq:local-h-3-pre}.
Then $X_{\hbar,\tau}^\ell(\g)$ together with following fields
\begin{align*}
  Y(h_i,z),\quad \frac{2r_i\hbar}{q_i-q_i\inv} Y(x_i^+,z),\quad Y(x_i^-,z)
\end{align*}
is an object of $\mathcal M$ and $X_{\hat\g,\hbar}(\ell,0)$ equipped with
\begin{align*}
  Y(h_{i,\hbar},z),\quad \frac{q_i-q_i\inv}{2r_i\hbar}Y(x_i^+,z),\quad
  Y(x_i^-,z)
\end{align*}
is an object of $\mathcal M_{\hbar,\tau}^\ell(\g)$.
By applying \cite[Proposition 5.4]{K-Quantum-aff-va} and Proposition \ref{prop:universal}, we get an $\hbar$-adic nonlocal vertex algebra isomorphism
$\psi:V(\mathcal M_{\hbar,\tau}^\ell(\g))\to F_\tau(A,\ell)$ determined by
\begin{align*}
  h_i\mapsto h_{i,\hbar},\quad x_i^+\mapsto \frac{q_i-q_i\inv}{2r_i\hbar}x_{i,\hbar}^+,\quad x_i^-\mapsto x_i^-\quad\te{for }i\in I.
\end{align*}
Note that
\begin{align*}
  &\frac{\tau_{ii}^{\bullet,\bullet}(-2r\ell\hbar)}
    {\tau_{ii}^{\bullet,\bullet}(2r\ell\hbar)}
  =\exp\left(f_0(-2r_i\hbar-2r\ell\hbar)
    -f_0(-2r_i\hbar+2r\ell\hbar)\right)\\
  =&\frac{q^{-r_i-r\ell}-q^{r_i+r\ell}}{-2r_i\hbar-2r\ell\hbar}
  \frac{-2r_i\hbar+2r\ell\hbar}{q^{-r_i+r\ell}-q^{r_i-r\ell}}
  =\frac{F(-r_i-r\ell)}{F(-r_i+r\ell)}
  =\frac{F(r_i+r\ell)}{F(r_i-r\ell)}.
\end{align*}
Then the relations \eqref{eq:x+0x-}, \eqref{eq:x+1x-},
\eqref{eq:serre} and \eqref{eq:def-integrable}
are exactly the relations \eqref{eq:x+0x-'}, \eqref{eq:x+1x-'}, \eqref{eq:serre'} and \eqref{eq:def-integrable'}, respectively.
Therefore, $\psi$ induces the desired isomorphism.
\end{proof}

\end{appendices}




\end{document}